\documentclass{amsart}

\usepackage[english]{babel}
\usepackage{changepage,geometry}
\usepackage{amsmath,amsthm,amssymb,stmaryrd,dsfont,mathrsfs}
\usepackage{hyperref}
\usepackage[scr=esstix]{mathalpha}
\usepackage{picture,subfigure,subcaption,float,svg,pdflscape}
\usepackage{array,multirow,arydshln}
\usepackage{tikz}
\usepackage{tikz-cd}
\usepackage{pgfplots}
\pgfplotsset{compat=1.18}
\usepackage{todonotes}

\newtheorem{theorem}{Theorem}[section]
\newtheorem{lemma}[theorem]{Lemma}
\newtheorem{corollary}[theorem]{Corollary}
\newtheorem{proposition}[theorem]{Proposition}

\theoremstyle{remark}
\newtheorem*{remark}{Remark}

\theoremstyle{definition}
\newtheorem{definition}[theorem]{Definition}
\newtheorem{example}[theorem]{Example}

\newtheorem{theoremIntro}{Theorem}

\allowdisplaybreaks

\newcommand{\R}{\mathbb{R}}

\newcommand{\C}{\mathbb{C}}

\newcommand{\cC}{\mathcal{C}}

\newcommand{\cF}{\mathcal{F}}

\newcommand{\cH}{\mathcal{H}}

\newcommand{\bE}{{\boldsymbol{E}}}
\newcommand{\bH}{{\boldsymbol{H}}}
\newcommand{\bGamma}{{\boldsymbol{\Gamma}}}

\newcommand{\im}{\mathrm{im}}

\newcommand{\End}{\mathrm{End}}

\newcommand{\Id}{\mathrm{Id}}
\newcommand{\tr}{\mathrm{tr}}

\newcommand{\rk}{\mathrm{rk}}
\newcommand{\Vol}{\mathrm{Vol}}

\newcommand{\scal}[2]{\left<#1,#2\right>}
\newcommand{\abs}[1]{\left|#1\right|}
\newcommand{\norme}[1]{\left\|#1\right\|}

\renewcommand{\leq}{\leqslant}
\renewcommand{\geq}{\geqslant}

\newcommand{\tend}[2]{\underset{#1 \rightarrow #2}{\longrightarrow}}

\newcommand{\fonction}[4]{
\left\{
\begin{array}{ccc}
    \displaystyle #1 & \longrightarrow & \displaystyle #2\\
    \displaystyle #3 & \longmapsto & \displaystyle #4
\end{array}
\right.
}

\newcommand{\NN}{\mathrm{N}}
\newcommand{\2}{\mathrm{I\!I}}

\title[Non-Linear Non-Spectral Elliptic Hessian Equations on Riemannian Manifolds]{On Polynomial Complex Hessian Equations I:\\
Non-Linear Non-Spectral Elliptic Hessian Equations\\ on Riemannian Manifolds}
\author[Rémi Delloque]{Rémi Delloque}
\address{Rémi Delloque, Department of Mathematical Sciences, Chalmers University of Technology and University of Gothenburg, SE-412 96 Gothenburg, Sweden}
\email{delloque@chalmers.se}
\date{}

\begin{document}

\maketitle

\begin{abstract}
    We study a general class of fully non-linear elliptic Hessian equations on compact Riemannian manifolds. In particular, we replace the standard assumption that the equation is a symmetric function of the eigenvalues of the Hessian made by Caffarelli, Nirenberg and Spruck \cite{CNS} by a weaker one. After deriving \textit{a priori} estimates for solutions under suitable assumptions on the equation, we solve it with a continuity method.
\end{abstract}

\section{Introduction}

\paragraph{\textbf{Motivations.}} The purpose of this series of articles is to give a numerical criterion to solve polynomial Hessian equations on a compact Kähler manifold $X$ of the form
\begin{equation}\label{EQ:Equation P-critique}
    \sum_{k = 0}^n \gamma_{n - k}\frac{\chi^k}{k!} = 0,
\end{equation}
where the $\gamma_k$ are closed real $(n - k,n - k)$-forms and $\chi$ is a closed real $(1,1)$-form of fixed cohomology class. It can be interpreted as a particular $P$-critical equation \cite{DNST} on a line bundle and it includes famous complex Hessian equations such that the Monge--Ampère equation \cite{Yau}\cite{Collins_Shi}, the $J$-equation \cite{Donaldson}\cite{Chen} and the deformed Hermitian Yang--Mills equation \cite{LYZ}\cite{Chen}.

While preparing this article, the case where the coefficients of the equation have the form
$$
\gamma_k = c_k\frac{\omega^k}{k!}
$$
where $c_k$ is a constant (except maybe $c_n$) and $\omega$ is a Kähler form has been solved by Fu, Li, D. Zhang and Z. Zhang \cite{FLZZ} (at least when $c_0 \neq 0$). By ``solved", we mean that an algebro-geometric criterion of resolution has been found. We shall refer to this particular kind of equation as a \textit{spectral equation} since it can be expressed as a symmetric equation in the eigenvalues of $\chi$ with respect to $\omega$.

The only known example of solved non-spectral equation of the form \eqref{EQ:Equation P-critique} was solved by Fang and Ma \cite{Fang_Ma}. Their main assumption on the coefficients $\gamma_k$ is that $\gamma_0 > 0$ and the other $\gamma_k$ are weakly non-positive.

The aim of this series of articles is to find suitable assumptions on the coefficients $\gamma_k$ that cover both the Fu--Li--Zhang--Zhang and the Fang--Ma equations. Since these equations are non-linear and elliptic, the first step is to find a criterion to solve general non-linear elliptic equations. Therefore, we first develop the analytic framework needed for polynomial complex Hessian equations by proving a general existence theorem for non-spectral Hessian equations on compact Riemannian manifolds. The study of non-linear elliptic spectral equations goes back to the work of Caffarelli, Nirenberg and Spruck \cite{CNS} and was followed by \cite{Chou_Wang}\cite{Hou_Ma_Wu}\cite{Tosatti_Weinkove}\cite{Szekelyhidi} for example. These equations take the form
$$
\forall x \in X, \quad f(\lambda(\nabla^2u(x) + A(x))) = h(x).
$$
Here, $u$ is the unknown, $\nabla^2u(x) : T_xX \rightarrow T_xX$ its Hessian matrix at $x$, $A(x)$ another symmetric endomorphism of $T_xX$ (usually the identity), $h(x) \in \R$, $\lambda(A)$ denotes the eigenvalues of a symmetric endomorphism $A$ and $f$ is a concave elliptic function of $\R^n$, invariant under permutation of the variables. Under suitable assumptions, such an equation admits a solution if and only if it admits a \textit{$\cC$-sub-solution} \cite{Guan} (see also Definition \ref{DEF:Sous-solution}).

It is interesting to notice that the notion of $\cC$-sub-solution is pointwise and only depends on the level-sets $\{f = h(x)\}$. Therefore, in this paper, we don't define the equation we want to solve as a function of the Hessian but directly as the level set we want to reach, which have to be \textit{elliptic convex hypersurfaces} (see Definition \ref{DEF:Hypersurface Gamma-elliptique} below). This point of view is new and is the main contribution of this paper to the general theory of non-linear elliptic Hessian equations on Riemannian manifolds. It allows more flexibility in the elliptic equations one can solve when we have a sub-solution, even in the spectral case.

Moreover, we want to remove the assumption that the equation is spectral. However, the only known example of non-spectral elliptic Hessian equation for which this result was shown is the Fang--Ma equation \cite{Fang_Ma}. The main difficulty when we remove the spectrality assumption is that we lose helpful concavity terms. To solve this issue, we introduce the notion of \textit{almost spectral} equation (see Definition \ref{DEF:Presque spectrale}).

Once the main definitions are given, the principal issue is to obtain \textit{a priori} estimates for solutions of the equation under consideration. The $L^1$ estimate is very standard and follows the methods of Tosatti and Weinkove \cite{Tosatti_Weinkove}. The $\cC^0$ estimate is also standard and uses Blocki's \cite{Blocki} version of the Alexandroff--Bakelman--Pucci maximum principle. Usually for this kind of equations, the $\cC^1$ estimate is obtained as a consequence of the $\cC^2$ estimate and a Liouville like theorem. Here, we get directly the $\cC^1$ estimate by introducing an auxiliary function and by studying what happens at a point where the maximum of such a function is reached. The $\cC^2$ estimate is also obtained thanks to an auxiliary function but this time, it is rather standard and follows the method of many authors \cite{Chou_Wang}\cite{Hou_Ma_Wu}\cite{Tosatti_Weinkove}\cite{Szekelyhidi}. The additional concavity terms that we get thanks to the almost spectrality assumption are used for the $\cC^2$ estimate.

\paragraph{\textbf{$\Gamma$-elliptic convex hypersurfaces.}} Let $(X,g)$ be a compact Riemannian manifold of dimension $n$. Let $\bH \subset \End_S(TX,g)$ a smooth hypersurface of (the total space of) the bundle of $g$-symmetric endomorphisms of $TX$. For all sub-sets $U$ of $X$, we write $\bH|_U = \bH \cap \End_S(TX|_U,g)$, the restriction of $\bH$ over $U$ and $H_x = \bH|_{\{x\}} = \bH \cap \End_S(T_xX,g)$ when $x \in X$.

The purpose of this article is to give sufficient conditions to solve the general Hessian equation
\begin{equation}\label{EQ:Equation}
    \forall x \in X, \quad \nabla^2u(x) \in H_x,
\end{equation}
where $\nabla^2u$ is the Hessian endomorphism of $u$ with respect to $g$, seen as a section of $\End_S(TX,g)$.

\begin{definition}\label{DEF:Hypersurface Gamma-elliptique}
    Let $(E,\scal{\cdot}{\cdot})$ be a real Euclidean vector space and $\Gamma \subset E$ a non-empty open cone with strictly convex closure. We say that $H \subset E$ is a \textit{smooth $\Gamma$-elliptic convex hypersurface} of $E$ if
    \begin{enumerate}
        \item $H$ is a non-empty connected closed smooth hypersurface of $E$.
        \item For all $A \in H$, the unit normal vector $\NN_H(A)$ at $A$ lies in $\Gamma^\lor$.
        \item For all $A \in H$, the second fundamental form $\2_H(A)(\hat{A},\hat{A}) = -\scal{d\NN_H(A)\hat{A}}{\hat{A}}$ is non-negative on $T_AH$.
    \end{enumerate}
\end{definition}

\begin{remark}
    Notice that the second point is equivalent to saying that if $\varphi : E \rightarrow \R$ is a linear form whose kernel is $T_AH$, then all the $\varphi(P)$ for $P \in \overline{\Gamma}\backslash\{0\}$ are positive (up to replacing $\varphi$ by $-\varphi$). In particular, this condition doesn't depend on the choice of the scalar product. Similarly, the third point is equivalent to saying that when $B \in H$ tends toward $A$, $\varphi(B) \geq \mathrm{o}(\|B - A\|^2)$, which is also independent from the scalar product.
\end{remark}

\begin{remark}
    If $\NN$ is a unit normal vector of $H$ at some point $A$, so is $-\NN$. Since $\Gamma$ is open and non-empty, $\Gamma^\lor$ has a strictly convex closure so $\NN$ and $-\NN$ can't both lie in this cone. Therefore, the first point in Definition \ref{DEF:Hypersurface Gamma-elliptique} ensures a way to choose canonically the normal unit vector. In particular, the second fundamental form is also well-defined globally and $H$ must be orientable.
\end{remark}

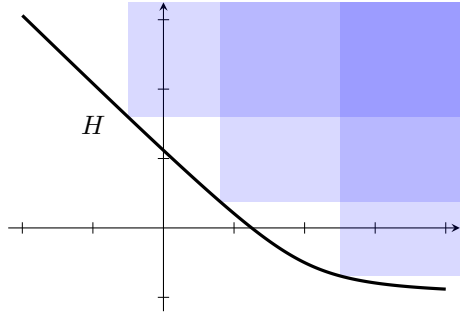
\begin{figure}[H]
\begin{center}
\begin{tikzpicture}

  \begin{axis}[
    axis lines=middle,
    xmin=-2.2, xmax=4.2,
    ymin=-1.2, ymax=3.25,
    width=0.5\linewidth,
    height=0.375\linewidth,
    samples=200,
    domain=-2:4,
    xtick={-2,-1,0,1,2,3,4},
    ytick={-1,0,1,2,3,4},
    xticklabels={},
    yticklabels={},
    tick style={black},
    axis line style={black},
    clip=false
  ]


    \pgfmathsetmacro{\xone}{-0.5}
    \pgfmathsetmacro{\yone}{0.5*(-\xone + sqrt((\xone-2)^2 + 1))}

    \fill[
      blue,
      opacity=0.15
    ]
    (axis cs:\xone,\yone)
      rectangle
    (axis cs:4.2,3.25);

    \pgfmathsetmacro{\xtwo}{0.8}
    \pgfmathsetmacro{\ytwo}{0.5*(-\xtwo + sqrt((\xtwo-2)^2 + 1))}

    \fill[
      blue,
      opacity=0.15
    ]
    (axis cs:\xtwo,\ytwo)
      rectangle
    (axis cs:4.2,3.25);

    \pgfmathsetmacro{\xthree}{2.5}
    \pgfmathsetmacro{\ythree}{0.5*(-\xthree + sqrt((\xthree-2)^2 + 1))}

    \fill[
      blue,
      opacity=0.15
    ]
    (axis cs:\xthree,\ythree)
      rectangle
    (axis cs:4.2,3.25);
    
    \addplot[
      black,
      very thick,
      domain=-2:4
    ]
    {0.5*(-x + sqrt((x-2)^2 + 1))}
    node[near start,below left] {$H$};

  \end{axis}
\end{tikzpicture}
\end{center}
\caption{An example of a $\Gamma$-elliptic convex hypersurface. Here, $E = \R^2$ and $\Gamma = (\R_+^*)^2$.}
\end{figure}

In Section \ref{SEC:Géométrie convexe}, we give the main properties of these hypersurfaces. In particular, we build functions $F_P$ that depend on a parameter $P \in \Gamma$ whose zero locus is $H$ and that we use to construct solutions of \eqref{EQ:Equation} with a continuity method.

We show that $H$ cuts $E$ into a bottom left part $H^-$ and a convex upper right part $H^+$. We also show Guan's inequality \cite{Guan} in this context, which is very helpful to show the $\cC^2$ estimates for \eqref{EQ:Equation}.

Then, in Sub-section \eqref{SEC:Famille}, we define \textit{families of $\bGamma$-elliptic hypersurfaces} (see Definition \ref{DEF:Famille}). When $\bE = \End_S(TX,g)$ is the bundle of $g$-symmetric endomorphisms of the tangent bundle and $\bGamma = \End_S^{++}(TX,g)$ is the open sub-set of positive definite endomorphisms, we define \textit{almost spectral families} (see Definition \ref{DEF:Presque spectrale}). It is a technical assumption that depends on a tuple $\underline{M}$ of positive constants and that allows us to control the derivatives of a function $F : \End_S(TX,g) \rightarrow \R$ whose zero locus is $\bH$ and to obtain additional concavity terms (see Lemma \ref{LEM:Conséquences hypothèse}).

With these assumptions, we can control the solutions of \eqref{EQ:Equation} in function of explicit parameters.

\paragraph{\textbf{Main results.}} First, we show \textit{a priori} estimates for the elliptic equation \eqref{EQ:Equation} in Section \ref{SEC:Estimées}.

In the next theorem, $\bH_0$ is a family of $\End_S^{++}(TX,g)$-convex hypersurfaces such that for all $x$, $H_{0,x}$ is an affine hyperplane contained in $H_x^-$ (see Figure \ref{FIG:H_0}).

\begin{theoremIntro}\label{THEI:Estimées C2}
    There is a constant $C = C(X,g,\bH_0,\delta,R,\underline{M})$ such that a solution $u$ of \eqref{EQ:Equation} such that $\max_X u = 0$ verifies
    $$
    \norme{u}_{\cC^2} \leq C.
    $$
\end{theoremIntro}

Once the $\cC^2$ norm is controlled, the Evans--Krylov theory \cite{Evans}\cite{Krylov}\cite{Caffarelli_Cabré} and the elliptic regularity theory give us bounds on the derivatives of $u$ of arbitrary order. With these \textit{a priori} estimates, we set up a continuity method and prove the following existence results in Section \ref{SEC:Méthode continuité}.

\begin{theoremIntro}\label{THEI:Existence solution 1}
    Let $\bH \subset \End_S(TX,g)$ be an almost spectral family of smooth $\End_S^{++}(TX,g)$-elliptic convex hypersurfaces of $\End(TX,g)$ and assume there exists a $\cC$-sub-solution $\underline{u} : X \rightarrow \R$. Assume moreover that for all smooth $u,h : X \rightarrow \R$, we have
    $$
    \forall x \in X, \quad \nabla^2u(x) - h(x)\Id \in H_x \quad \Longrightarrow \quad \int_X h \, \Vol_g = 0.
    $$
    
    Then, there is a smooth $u : X \rightarrow \R$, unique up to a constant, such that for all $x \in X$,
    $$
    \nabla^2u(x) \in H_x.
    $$
\end{theoremIntro}

The additional assumption in Theorem \ref{THEI:Existence solution 1} can be obtained by a cohomological argument for some complex Hessian equations. However, it needs not hold in general. The last theorem gives an existence result of solutions of the Hessian equation
$$
\nabla^2u(x) \in H_x^c
$$
where $(\bH^c)_{a \leq c \leq b}$ are suitable almost spectral families of smooth $\End_S^{++}(TX,g)$-elliptic convex hypersurfaces and $c$ is an additional unknown.

\begin{theoremIntro}\label{THEI:Existence solution 2}
    Let $a < b$ be real numbers. Let $\boldsymbol{\cH}$ be a family of smooth $(\End_S^{++}(TX,g) \times [a,b])$-elliptic convex hypersurfaces in the bundle $\End_S(TX,g) \times [a,b] \rightarrow X \times [a,b]$. In particular, for all $c \in [a,b]$,
    $$
    \bH^c = \{(x,A,t) \in \boldsymbol{\cH} \mid t = c\}
    $$
    is a family of smooth $\End_S^{++}(TX,g)$-elliptic convex hypersurfaces on $\End_S(TX,g)$. We assume that there is a uniform tuple $\underline{M}$ such that each $\bH^c$ is almost spectral with tuple $\underline{M}$. We also assume that the map $c \mapsto (H_p^c)^-$ is non-decreasing for the inclusion.
    
    We assume moreover that there is a smooth function $\underline{u} : X \rightarrow \R$ which is a $\cC$-sub-solution for $\bH^b$ and that for all $x \in X$,
    $$
    \nabla^2\underline{u}(x) \in (H_x^a)^+ \cap (H_x^b)^- = \bigcup_{c \in [a,b]} H_x^c.
    $$
    
    Then, there is a smooth $u : X \rightarrow \R$ unique up to a constant and a $c \in [a,b]$ such that
    $$
    \forall x \in X, \quad \nabla^2u(x) \in H_x^c.
    $$
    If in addition, for all $c_1 \neq c_2$, $\bH^{c_1} \cap \bH^{c_2} = \emptyset$, then $c$ is unique too.
\end{theoremIntro}

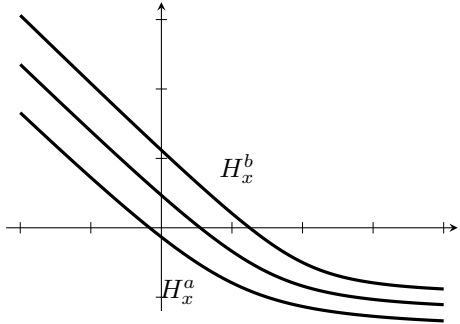
\begin{figure}[H]
\begin{center}
\begin{tikzpicture}
  \begin{axis}[
    axis lines=middle,
    xmin=-2.2, xmax=4.2,
    ymin=-1.2, ymax=3.25,
    width=0.5\linewidth,
    height=0.375\linewidth,
    samples=200,
    domain=-2:4,
    xtick={-2,-1,0,1,2,3,4},
    ytick={-1,0,1,2,3,4},
    xticklabels={},
    yticklabels={},
    tick style={black},
    axis line style={black},
    clip=false
  ]

    \addplot[
      black,
      very thick,
      domain=-2:4
    ]
    {0.5*(-x + sqrt((x-2)^2 + 1))}
    node[midway,above right] {$H_x^b$};

    \addplot[
      black,
      very thick,
      domain=-2:4
    ]
    {0.5*(-x + sqrt((x-1.5)^2 + 1.5)) - 0.5};

    \addplot[
      black,
      very thick,
      domain=-2:4
    ]
    {0.5*(-x + sqrt((x-1)^2 + 2)) - 1}
    node[midway,below left] {$H_x^a$};

  \end{axis}
\end{tikzpicture}
\end{center}
\caption{Example of the $H_x^c$ with $x$ fixed when $c$ varies.}
\end{figure}

\paragraph{\textbf{Acknowledgment}} I thank Carl Tipler for his support and for helpful discussions about this problem. I also thank Lars Martin Sektnan for his feedback and for helpful discussions.

\paragraph{\textbf{AI use}} AI was used marginally to correct spelling mistakes, polish the introduction and generate the LaTex code for the figures. I am responsible for all mathematical contents and assume full responsibility in case of errors.

\section{Convex geometry}\label{SEC:Géométrie convexe}

In this section, we fix $(E,\scal{\cdot}{\cdot})$ a real Euclidean space, $\Gamma \subset V$ a non-empty open cone with strictly convex closure and $H \subset V$ a smooth $\Gamma$-elliptic convex hypersurface (see Definition \ref{DEF:Hypersurface Gamma-elliptique}). We show the main properties of these hypersurfaces. In particular, for each $P \in \Gamma$, we construct smooth functions $F_P : E \rightarrow \R$ whose zero locus is $H$ and we study the properties of these functions. We show that they are concave and $\Gamma$-elliptic in the following sense.

\begin{definition}\label{DEF:Fonction Gamma-elliptique}
    A function $F : E \rightarrow \R$ is $\Gamma$-elliptic if for all $A \in E$, $\nabla F(A) \in \Gamma^\lor$.
\end{definition}

This section serves as a preparation to obtain \textit{a priori} estimates for \eqref{EQ:Equation}. We also use the globally defined function $F_P$ to set up a continuity method in Section \ref{SEC:Méthode continuité}.

\subsection{The functions \texorpdfstring{$F_P$}{FP}}\label{SEC:Fonction F_P}

In this sub-section, we define the functions $F_P$ and show their main properties.

\begin{lemma}\label{LEM:TFI construction F_P}
    Let $A \in E$, $t \in \R$ and $P \in V$ such that $A - tP \in H$ and $\scal{\NN_H(A - tP)}{P} > 0$. There is a unique continuous germ of functions
    $$
    F_{A,t,P} : U_1 \times U_2 \longrightarrow \R
    $$
    with $U_1$ an open neighbourhood of $A$ in $E$ and $U_2$ an open neighbourhood of $P$ in $V$ such that $F_{A,t,P}(A,P) = t$ and for all $(B,Q) \in U_1 \times U_2$, $B - F_{A,t,P}(B,Q)Q \in H$. Moreover, $F_{A,t,P}$ is smooth and its gradient with respect to $B \in U_1$ is given at the origin by
    $$
    \nabla_BF_{A,t,P}(A,P) = \frac{\NN_H(A - tP)}{\scal{\NN_H(A - tP)}{P}}.
    $$
\end{lemma}
\begin{proof}
It simply follows from the implicit function theorem. Since $H$ is a smooth hypersurface, there is a smooth function $G$ defined in an open neighbourhood $\tilde{U}$ of $A - tP$ in $E$ with real values such that $\nabla G(A - tP) \neq 0$ and $\tilde{U} \cap H = \{B \in \tilde{U} \mid G(B) = 0\}$. Moreover, up to multiplying $G$ by a non-zero constant, we may assume that $\nabla G(A - tP) = \NN_H(A - tP)$ is the unit normal vector in $\Gamma^\lor$.

By assumption, $\scal{\nabla G(A - tP)}{P} = \scal{\NN_H(A - tP)}{P} > 0$. Let $U_1 = \tilde{U} + tP$, which is an open neighbourhood of $A \in E$ and let the germ of smooth functions,
$$
\Psi : \fonction{(E,A) \times (E,P) \times (\R,t)}{(\R,0)}{(B,Q,s)}{G(B - sQ)}.
$$
We have $\Psi(A,P,t) = G(A - tP) = 0$ and $\frac{\partial\Psi}{\partial s}(A,P,t) = \scal{\nabla G(A - tP)}{-P} < 0$. The desired result follows directly from the implicit function theorem applied to $\Psi$. The gradient with respect to the first variable is given by
$$
\nabla_BF_{A,t,P}(A,P) = -\frac{\nabla_B\Psi(A,P,t)}{\frac{\partial\Psi}{\partial s}(A,P,t)} = \frac{\NN_H(A - tP)}{\scal{\NN_H(A - tP)}{P}}.
$$
\end{proof}

Fix $A_0 \in H$ and $I \in \Gamma$. For later use, we introduce
\begin{equation}\label{EQ:c(I)}
    c(I) = \inf_{R \in \Gamma^\lor, \norme{R} = 1} \scal{R}{I} > 0.
\end{equation}
We define $r_{\sup}$ to be the supremum of all positive real numbers $r$ such that there is a continuous function $F_I : B(A_0,r) \rightarrow \R$ such that $F_I(A_0) = 0$ and for all $x \in B(A_0,r)$,
\begin{equation}\label{EQ:Définition F_I}
    A - F_I(A)I \in H.
\end{equation}

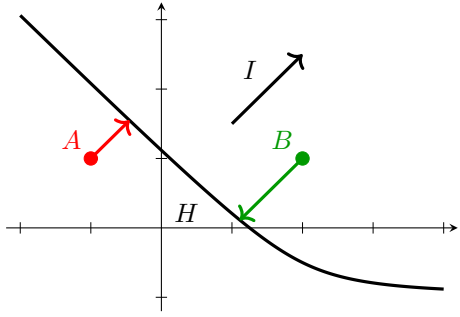
\begin{figure}[H]
\begin{center}
\begin{tikzpicture}
  \begin{axis}[
    axis lines=middle,
    xmin=-2.2, xmax=4.2,
    ymin=-1.2, ymax=3.25,
    width=0.5\linewidth,
    height=0.375\linewidth,
    samples=200,
    domain=-2:4,
    xtick={-2,-1,0,1,2,3,4},
    ytick={-1,0,1,2,3,4},
    xticklabels={},
    yticklabels={},
    tick style={black},
    axis line style={black},
    clip=false
  ]

    \addplot[
      black,
      very thick,
      domain=-2:4
    ]
    {0.5*(-x + sqrt((x-2)^2 + 1))}
    node[midway,below left] {$H$};

    \addplot[
      only marks,
      mark=*,
      mark size=2.5pt,
      red
    ]
    coordinates {(-1,1)}
    node[above left] {$A$};

    \draw[
      ->,
      red,
      very thick
    ]
    (axis cs:-1,1) -- (axis cs:-0.451,1.549);

    \addplot[
      only marks,
      mark=*,
      mark size=2.5pt,
      green!60!black
    ]
    coordinates {(2,1)}
    node[above left] {$B$};

    \draw[
      ->,
      green!60!black,
      very thick
    ]
    (axis cs:2,1) -- (axis cs:1.112,0.112);

    \draw[
      ->,
      black,
      very thick
    ]
    (axis cs:1,1.5) -- (axis cs:2,2.5)
    node[
      midway,
      above left
    ] {$I$};

  \end{axis}
\end{tikzpicture}
\end{center}
\caption{Illustration of the function $F_I$. Here, $F_I(A) < 0$ and $F_I(B) > 0$. $F_I = 0$ on $H$.}
\end{figure}

We want to show that $r_{\sup} = +\infty$ and $F_I$ is unique with this property. In particular, we show that it is well-defined on the whole $E$. Since $I \in \Gamma$ and by assumption on $H$, for all $A \in H$, we have $\scal{\NN_H(A)}{I} > 0$ so we will be able to use Lemma \ref{LEM:TFI construction F_P} with $P = I$.

\begin{lemma}\label{LEM:Unicité F_I}
    If $r > 0$ and $F_I,\tilde{F}_I : B(A_0,r) \rightarrow \R$ are continuous functions with $F_I(A_0) = \tilde{F}_I(A_0) = 0$ and they both satisfy \eqref{EQ:Définition F_I}, then $F_I = \tilde{F}_I$. Moreover, $F_I$ is smooth and for all $A \in B(A_0,r)$,
    $$
    \nabla F_I(A) = \frac{\NN_H(A - F_I(A)I)}{\scal{\NN_H(A - F_I(A)I)}{I}}
    $$
\end{lemma}
\begin{proof}
Let $A \in B(0,r)$. For all $B \in B(A_0,r)$ close enough to $A$, $B - F_I(B)I \in H$ and by continuity of $F_I$ at $A$ and Lemma \ref{LEM:TFI construction F_P}, we must have $F_I(B) = F_{A,I,F_I(A)}(B,I)$. In particular, $F_I$ is smooth near $A$ and
$$
\nabla F_I(A) = \frac{\NN_H(A - F_I(A)I)}{\scal{\NN_H(A - F_I(A)I)}{I}}.
$$

Symmetrically, $\tilde{F}_I$ is smooth near $A$ and
$$
\nabla\tilde{F}_I(A) = \frac{\NN_H(A - \tilde{F}_I(A)I)}{\scal{\NN_H(A - \tilde{F}_I(A)I)}{I}}.
$$
This shows that both $F_I$ and $\tilde{F}_I$ satisfy the same first order differential equation. Since moreover $F_I(A_0) = \tilde{F}_I(A_0) = 0$ and $B(A_0,r)$ is connected, we must have $\tilde{F}_I = F_I$.
\end{proof}

\begin{proposition}\label{PRO:F_I bien définie}
    We have $r_{\sup} = +\infty$ and $F_I$ is defined on the whole $E$. Moreover, for all $(A,t)$ such that $A - tI \in H$, we have $t = F_I(A)$. In other words, each affine line of direction $I$ crosses $H$ exactly once.
\end{proposition}
\begin{proof}
First of all, by Lemma \ref{LEM:TFI construction F_P} applied at $A_0$, we must have $r_{\sup} > 0$ with $F_I = F_{A_0,I,0}(\cdot,I)$ near $A_0$. Assume from now on that $r_{\sup} < +\infty$.

By the uniqueness result of Lemma \ref{LEM:Unicité F_I}, $F_I$ must in particular be defined on the whole $B(A_0,r_{\sup})$. By Lemma \ref{LEM:Unicité F_I} and \eqref{EQ:c(I)}, for all $A \in B(A_0,r_{\sup})$,
$$
\norme{\nabla F_I(A)} = \norme{\frac{\NN_H(A - F_I(A)I)}{\scal{\NN_H(A - F_I(A)I)}{I}}} \leq \frac{\norme{\NN_H(A - F_I(A)I)}}{c(I)} = \frac{1}{c(I)}.
$$
This shows that $F_I$ is globally Lipschitz so it extends to a continuous function on $\overline{B}(A_0,r_{\sup})$. Moreover, since $H$ is closed in $E$ by Definition \ref{DEF:Hypersurface Gamma-elliptique}, each $A - F_I(A)I$ is in $H$ even when $A$ is in the sphere $S(A_0,r_{\sup})$.

By applying Lemma \ref{LEM:TFI construction F_P} at each point of $S(A_0,r_{\sup})$ and by gluing the functions we obtain thanks to Lemma \ref{LEM:Unicité F_I}, $F_I$ can be extended in an open neighbourhood $U$ of $\overline{B}(A_0,r_{\sup})$. It is a standard result that $U$ must contain an open ball of centre $A_0$ and radius $r > r_{\sup}$. It contradicts the maximality of $r_{\sup}$ hence $r_{\sup} = +\infty$. As above, $F_I$ is well-defined on the whole $E$.

We have left to show that for all $A$ and $t$ with $A - tI \in H$, $t = F_I(A)$. Let
$$
\tilde{H} = \{A \in H \mid F_I(A) = 0\}.
$$
Clearly, $\tilde{H}$ is a closed sub-set of $H$ and it is non-empty since $A_0 \in \tilde{H}$. By Lemma \ref{LEM:TFI construction F_P}, $\tilde{H}$ is also open in $H$. Recall that by Definition \ref{DEF:Hypersurface Gamma-elliptique}, $H$ is connected hence $\tilde{H} = H$.

Therefore, if $(A,t) \in E \times \R$ is such that $A - tI \in H$, we have $F_I(A - tI) = 0$. By Lemma \ref{LEM:Unicité F_I}, for all $B \in E$, $\scal{\nabla F_I(B)}{I} = 1$. We deduce that
$$
F_I(A) = F_I(A - tI) + t = t.
$$
\end{proof}

\begin{remark}
    A consequence of Proposition \ref{PRO:F_I bien définie} is that $F_I$ is fully characterised by the equality $A - F_I(A)I \in H$ for all $A$. In particular, $F_I$ is independent of the choice of $A_0$.
\end{remark}

Let
$$
H^+ = \{A \in E \mid F_I(A) \geq 0\}, \qquad H^- = \{A \in E \mid F_I(A) \leq 0\}.
$$
By Proposition \ref{PRO:F_I bien définie}, $H = H^+ \cap H^-$. It is easy to see that the sets $H^+$ and $H^-$ are independent from the choice of $I \in \Gamma$. Indeed, if $P \in \Gamma$ and $A \in V$ is such that $F_I(A) \geq 0$, then by Lemma \ref{LEM:Unicité F_I},
\begin{align*}
    F_P(x) & = F_P(A - F_I(A)I) + \int_0^{F_I(A)} \scal{\nabla F_P(A - sI)}{I} \, ds\\
    & = \int_0^{F_I(x)} \frac{\scal{\NN_H(A - sI - F_P(A - sI)P)}{I}}{\scal{\NN_H(A - sI - F_P(A - sI)P)}{P}} \, ds\\
    & \geq 0.
\end{align*}

We used the fact that $F_P(A - F_I(x)I) = 0$ because $A - F_I(A) \in H$ and the integrand is positive because both $I$ and $P$ lie in $\Gamma$ and the normal vector at each point of $H$ lies in $\Gamma^\lor$. We deduce that $F_P(A) \geq 0$. Symmetrically, we see that the sign of $F_P$ is always the same as the sign of $F_I$.

\subsection{Convexity and the set \texorpdfstring{$\Omega_H$}{OmegaH}}

In this sub-section, we show convexity and concavity properties and we construct an open set $\Omega_H$ with important properties that will later characterise the space of \textit{$\cC$-sub-solutions} when we solve elliptic PDEs.

We start by computing the second derivative of $F_I$. Recall that for all $(A,\hat{A}) \in TE = E^2$,
$$
dF_I(A)\hat{A} = \scal{\nabla F_I(A)}{\hat{A}} = \frac{\scal{\NN_H(A - F_I(A)I)}{\hat{A}}}{\scal{\NN_H(A - F_I(A)I)}{I}}.
$$
Let $B = A - F_I(A)I \in H$. We have,
\begin{adjustwidth}{-5cm}{5cm}
\begin{align}\label{EQ:Hessienne F_I}
    d^2F_I(A)(\hat{A},\hat{A}) & = \frac{\scal{d\NN_H(A - F_I(A)I)(\hat{A} - (dF_I(A)\hat{A})I)}{\hat{A}}}{\scal{\NN_H(A - F_I(A)I)}{I}} \nonumber\\
    & - \frac{\scal{\NN_H(A - F_I(A)I)}{\hat{A}}}{\scal{\NN_H(A - F_I(A)I)}{I}}\frac{\scal{d\NN_H(A - F_I(A)I)(\hat{A} - (dF_I(A)\hat{A})I)}{I}}{\scal{\NN_H(A - F_I(A)I)}{I}} \nonumber\\
    & = \frac{\scal{d\NN_H(A)\hat{A}}{\hat{A}}}{\scal{\NN_H(A)}{I}} - (dF_I(A)\hat{A})\frac{\scal{d\NN_H(A)I}{\hat{A}}}{\scal{\NN_H(A)}{I}} \nonumber\\
    & - (dF_I(A)\hat{A})\left(\frac{\scal{d\NN_H(A)\hat{A}}{I}}{\scal{\NN_H(A)}{I}} - (dF_I(A)\hat{A})\frac{\scal{d\NN_H(A)I}{I}}{\scal{\NN_H(A)}{I}}\right) \nonumber\\
    & = \frac{1}{\scal{\NN_H(A)}{I}}\left(-\2_H(A)(\hat{A},\hat{A}) + 2(dF_I(A)\hat{A})\2_H(A)(I,\hat{A}) - (dF_I(A)\hat{A})^2\2_H(A)(I,I)\right) \nonumber\\
    & = -\frac{\2_H(A - F_I(A)I)(\hat{A} - (dF_I(A)\hat{A})I,\hat{A} - (dF_I(A)\hat{A})I)}{\scal{\NN_H(A - F_I(A)I)}{I}}.
\end{align}
\end{adjustwidth}

\begin{proposition}\label{PRO:F_I concave}
    $F_I$ is concave.
\end{proposition}
\begin{proof}
It is an immediate consequence of \eqref{EQ:Hessienne F_I} and Definition \ref{DEF:Hypersurface Gamma-elliptique}.
\end{proof}

As a consequence, since $H^+$ is the $0$ super-level set of $F_I$, it is convex.

\begin{corollary}\label{COR:H^+ convexe}
    $H^+$ is convex.
\end{corollary}

Then, we define the set $\Omega_H$ that will be useful later to define $\cC$-sub-solutions.

\begin{definition}
    We denote by $\Omega_H$ the set of all $A \in E$ such that
    $$
    (A + \Gamma) \cap H
    $$
    is bounded.
\end{definition}

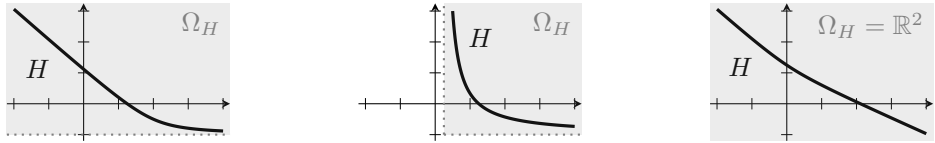
\begin{figure}[H]
\begin{center}
\begin{minipage}{0.3\textwidth}
\begin{tikzpicture}
  \begin{axis}[
    axis lines=middle,
    xmin=-2.2, xmax=4.2,
    ymin=-1.2, ymax=3.25,
    width=\linewidth,
    height=0.75\linewidth,
    samples=200,
    xtick={-2,-1,0,1,2,3,4},
    ytick={-1,0,1,2,3,4},
    xticklabels={},
    yticklabels={},
    tick style={black},
    axis line style={black},
    clip=false
  ]

    \addplot[
      black,
      very thick,
      domain=-2:4
    ]
    {0.5*(-x + sqrt((x-2)^2 + 1))}
    node[near start,below left] {$H$};
    
    \fill[
      gray,
      opacity=0.15
    ]
    (axis cs:-2.2,-1)
      rectangle
    (axis cs:4.2,3.25)
    node[opacity=1,below left] {$\Omega_H$};
    
    \draw[
      dotted,
      gray,
      thick
    ]
    (axis cs:-2.2,-1) -- (axis cs:4.2,-1);
    
  \end{axis}
\end{tikzpicture}
\end{minipage}
\begin{minipage}{0.3\textwidth}
\begin{tikzpicture}
  \begin{axis}[
    axis lines=middle,
    xmin=-2.2, xmax=4.2,
    ymin=-1.2, ymax=3.25,
    width=\linewidth,
    height=0.75\linewidth,
    samples=200,
    xtick={-2,-1,0,1,2,3,4},
    ytick={-1,0,1,2,3,4},
    xticklabels={},
    yticklabels={},
    tick style={black},
    axis line style={black},
    clip=false
  ]

    \addplot[
      black,
      very thick,
      domain=0.5:4
    ]
    {1/(x - 0.25) - 1}
    node[near start,above right] {$H$};
    
    \fill[
      gray,
      opacity=0.15
    ]
    (axis cs:0.25,-1)
      rectangle
    (axis cs:4.2,3.25)
    node[opacity=1,below left] {$\Omega_H$};
    
    \draw[
      dotted,
      gray,
      thick
    ]
    (axis cs:0.25,-1) -- (axis cs:4.2,-1);
    
    \draw[
      dotted,
      gray,
      thick
    ]
    (axis cs:0.25,-1) -- (axis cs:0.25,3.25);
    
  \end{axis}
\end{tikzpicture}
\end{minipage}
\begin{minipage}{0.3\textwidth}
\begin{tikzpicture}
  \begin{axis}[
    axis lines=middle,
    xmin=-2.2, xmax=4.2,
    ymin=-1.2, ymax=3.25,
    width=\linewidth,
    height=0.75\linewidth,
    samples=200,
    xtick={-2,-1,0,1,2,3,4},
    ytick={-1,0,1,2,3,4},
    xticklabels={},
    yticklabels={},
    tick style={black},
    axis line style={black},
    clip=false
  ]

    \addplot[
      black,
      very thick,
      domain=-2:4
    ]
    {0.25*(sqrt(x^2 + 1) - 3*x) + 1}
    node[near start,below left] {$H$};
    
    \fill[
      gray,
      opacity=0.15
    ]
    (axis cs:-2.2,-1.2)
      rectangle
    (axis cs:4.2,3.25)
    node[opacity=1,below left] {$\Omega_H = \R^2$};
    
  \end{axis}
\end{tikzpicture}
\end{minipage}
\end{center}
\caption{Illustrations of the set $\Omega_H$.}
\end{figure}

\begin{lemma}\label{LEM:F_Q avec Q au bord}
    Let $Q \in \partial\Gamma\backslash\{0\}$ and $A \in V$.
    \begin{enumerate}
        \item If there is $t \in \R$ such that $A - tQ \in H$, then
        $$
        \lim_{P \rightarrow Q, P \in \Gamma} F_P(A) = t.
        $$
        \item If there is no $t \in \R$ such that $A - tQ \in H$, then
        $$
        \lim_{P \rightarrow Q, P \in \Gamma} F_P(A) = -\infty.
        $$
    \end{enumerate}
    Moreover, the set of all $(A,Q)$ such that point 1 holds is open in $E \times \partial\Gamma\backslash\{0\}$.
\end{lemma}
\begin{proof}
We start with the first point. Let $t \in \R$ such that $A - tQ \in H$. Since $Q \in \overline{\Gamma}$ and $\NN_H(A - tQ) \in \Gamma^\lor$, $\scal{\NN_H(A - tQ)}{Q} > 0$. By Lemma \ref{LEM:TFI construction F_P}, we can construct a smooth function $F_{A,Q,t}$ near $(A,Q)$ such that for all $(B,P)$ close enough to $(A,Q)$,
$$
B - F_{A,Q,t}(B,P)P \in H.
$$
By Proposition \ref{PRO:F_I bien définie}, if moreover $P \in \Gamma$, we must have $F_{A,Q,t}(B,P) = F_P(B)$. By continuity of $F_{A,Q,t}(A,\cdot)$ at $Q$,
$$
\lim_{P \rightarrow Q, P \in \Gamma} F_P(A) = t,
$$
hence the first point of the lemma. Moreover, for all $(B,P) \in E \times \partial\Gamma$ close enough to $(A,Q)$, $B - F_{A,Q,t}(B,P)P \in H$ hence the openness of this property.

We then show the contraposition of the second point. Assume that
\begin{equation}\label{EQ:limsup > -infty}
    \limsup_{P \rightarrow Q, P \in \Gamma} F_P(A) > -\infty.
\end{equation}
Let $A_0 \in H$. Since $\NN_H(A_0) \in \Gamma^\lor$ and $\norme{\NN_H(A_0)} = 1$, there is a constant $\varepsilon > 0$ such that for all $P \in \Gamma$,
\begin{equation}\label{EQ:Borne inférieure ps}
    \scal{\NN_H(A_0)}{P} \geq \varepsilon\norme{P}.
\end{equation}
This $\varepsilon$ is the analogue of $c(I)$ for $\NN_H(A_0)$. By concavity of $F_P$ (Proposition \ref{PRO:F_I concave}), we have
\begin{align*}
    F_P(A) & \geq \scal{\nabla F_P(A_0)}{A} + F_P(A_0)\\
    & = \frac{\scal{\NN_H(A_0)}{A}}{\scal{\NN_H(A_0)}{P}} \textrm{ by Lemma \ref{LEM:Unicité F_I}},\\
    & \geq -\frac{\abs{\scal{\NN_H(A_0)}{A}}}{\varepsilon\norme{P}} \textrm{ by \eqref{EQ:Borne inférieure ps}}.
\end{align*}
We deduce that
\begin{equation}\label{EQ:limsup < +infty}
    \limsup_{P \rightarrow Q, P \in \Gamma} F_P(A) < +\infty,
\end{equation}
since $Q \neq 0$.

By \eqref{EQ:limsup > -infty} and \eqref{EQ:limsup < +infty}, we can construct a converging sequence $P_m \rightarrow Q$ such that $(F_{P_m}(A))$ converges in $\R$. Let $t$ be the limit. By closedness of $H$, $A - tQ \in H$, which proves the second point of the lemma.
\end{proof}

\begin{proposition}\label{PRO:Caractérisation Omega_H}
    For all $A \in E$, $A \in \Omega_H$ if and only if for all $Q \in \partial\Gamma\backslash\{0\}$, the point 1 in Lemma \ref{LEM:F_Q avec Q au bord} holds.
\end{proposition}
\begin{proof}
If for some $Q \in \partial\Gamma\backslash\{0\}$, the point 2 holds, then by this lemma, $F_P(A) \rightarrow -\infty$ when $P \rightarrow Q$. In particular, $A \in H^-$ so all the $F_P(A)$ are non-positive thus for all $P \in \Gamma$, $A - F_P(A)P \in (A + \Gamma) \cap H$ and
$$
\norme{A - F_P(A)P} \geq -F_P(A)\norme{P} - \norme{A} \tend{P}{Q} +\infty.
$$
We deduce that $A \notin \Omega_H$.

Else, there is a constant $R > 0$ such that for all $P \in \Gamma$, if $x + P \in H$, then $\norme{x + P} \leq R$. In this case, for all $P$,
$$
-F_P(A) \leq \frac{\norme{F_P(A)P}}{\norme{P}} \leq \frac{\norme{A - F_P(A)P} + \norme{A}}{\norme{P}} \leq \frac{R + \norme{A}}{\norme{P}}.
$$
It shows that $F_P(A)$ is bounded from above when $P$ approaches any $Q \in \partial\Gamma\backslash\{0\}$. In particular, the point 2 of Lemma \ref{LEM:F_Q avec Q au bord} never holds, which proves the proposition.
\end{proof}

\begin{corollary}\label{COR:Propriétés Omega_H}
    $\Omega_H$ is open in $E$, convex and it contains $H^+$.
\end{corollary}
\begin{proof}
If $A \in H^+$, then for all $P$, $F_P(x) \geq 0$ so $x + P \notin H$. We deduce that $(A + \Gamma) \cap H = \emptyset$ is bounded so $H^+ \subset \Omega_H$.

For the openness, let $A \in \Omega_H$. By Proposition \ref{PRO:Caractérisation Omega_H}, for all $Q \in \partial\Gamma\backslash\{0\}$, the couple $(A,Q)$ verifies the point 1 of Lemma \ref{LEM:F_Q avec Q au bord}. Moreover, this property is open so there is a neighbourhood $U_Q \times W_Q$ of $(A,Q)$ in $E \times \partial\Gamma\backslash\{0\}$ such that each $(B,P) \in U_Q \times W_Q$ also verify the point 1 of Lemma \ref{LEM:F_Q avec Q au bord}.

Let $K = \{Q \in \partial\Gamma \mid \norme{Q} = 1\}$. $K$ is compact and is covered by $(W_Q)_{Q \in K}$ so there is a finite sub-set $J$ of $K$ such that
$$
K \subset \bigcup_{j \in J} W_j.
$$
Let $U = \bigcap_{j \in J} U_j$, which is an open neighbourhood of $A \in E$.

For all $B \in U$ and all $Q \in \partial\Gamma\backslash\{0\}$, there is $j \in J$ such that $\frac{Q}{\norme{Q}} \in W_j$. Since $B \in U_j$, there is a $t \in \R$ such that
$$
B - \frac{t}{\norme{Q}}Q \in H.
$$
It is true for all $Q$ so $B \in \Omega_H$ by Proposition \ref{PRO:Caractérisation Omega_H} hence the openness of $\Omega_H$.

Finally, to see that $\Omega_H$ is convex, let $(A,B) \in \Omega_H^2$ and $Q \in \partial\Gamma\backslash\{0\}$. By Proposition \ref{PRO:Caractérisation Omega_H}, there are real numbers $(s,t)$ such that
$$
A - sQ \in H, \quad B - tQ \in H.
$$
By Corollary \ref{COR:H^+ convexe}, $H^+$ is convex so
$$
\frac{A + B}{2} - \frac{s + t}{2}Q \in H^+.
$$
Since $H^+ \subset \Omega_H$, there is a real number $r$ such that
$$
\frac{A + B}{2} - \left(\frac{s + t}{2} + r\right)Q \in H.
$$
It holds for any $Q$ so by Proposition \ref{PRO:Caractérisation Omega_H}, $\frac{A + B}{2} \in \Omega_H$ hence the convexity of $\Omega_H$.
\end{proof}

\subsection{Guan's inequality}

In this section, we extend Guan's inequality to non-spectral functions \cite[Theorem 2.18]{Guan} (see also \cite[Proposition 6]{Szekelyhidi}). Let $A_0 \in \Omega_H$. By Corollary \ref{COR:Propriétés Omega_H}, $\Omega_H$ is open so there is $\delta > 0$ such that $A_0 - \delta I \in \Omega_H$. By Proposition \ref{PRO:Caractérisation Omega_H},
\begin{equation}\label{EQ:Définition R}
    R = -\inf_{P \in \Gamma, \norme{P} = 1} F_P(A_0 - \delta I)
\end{equation}
is finite.

Let $D$ be a convex open neighbourhood of $H$ in $E$ and $F : D \rightarrow \R$ a smooth $\Gamma$-elliptic (see Definition \ref{DEF:Fonction Gamma-elliptique}) concave function such that each level set of $F$ is closed in $E$ and
$$
H = \{A \in D \mid F(A) = 0\}.
$$
In practice, we will only use the next proposition with $F = F_I$ but we show it for any $F$ for completeness.

\begin{proposition}\label{PRO:Guan}
    There is $\kappa > 0$ that only depends on $\delta$ and $R$ such that for all $A \in D \cap H^-$,
    $$
    -\scal{\nabla F(A)}{A - A_0} \geq \kappa\scal{\nabla F(A)}{I}, \quad \textrm{or} \quad \inf_{P \in \Gamma, \norme{P} = 1} \scal{\nabla F(A)}{P} \geq \kappa\scal{\nabla F(A)}{I}.
    $$
    Moreover, in any case,
    $$
    -\scal{\nabla F(A)}{A - A_0} \geq -R\scal{\nabla F(A)}{I}.
    $$
\end{proposition}
\begin{figure}[H]
\begin{center}
\begin{tikzpicture}
  \begin{axis}[
    axis lines=middle,
    xmin=-2.2, xmax=4.2,
    ymin=-1.2, ymax=3.25,
    width=0.5\linewidth,
    height=0.375\linewidth,
    samples=200,
    xtick={-2,-1,0,1,2,3,4},
    ytick={-1,0,1,2,3,4},
    xticklabels={},
    yticklabels={},
    tick style={black},
    axis line style={black},
    clip=false
  ]

    \addplot[
      black,
      very thick,
      domain=-2:4
    ]
    {0.5*(-x + sqrt((x-2)^2 + 1))}
    node[midway,above right] {$H$};

    %
    \addplot[
      blue,
      domain=-2:2.162
    ]
    {
      0.5*(1 + sqrt(10))
      + (-0.5 - 3/(2*sqrt(10)))*(x + 1)
    };

    %
    \addplot[
      blue,
      domain=-2:4
    ]
    {
      (-2.8 + sqrt(1.64))/2
      + (-0.5 + 0.5/(sqrt(1.64)))*(x - 2.8)
    };

    \addplot[
      only marks,
      mark=*,
      mark size=2.5pt,
      blue
    ]
    coordinates {(-1,2.081)}
    node[above right] {$A$};

    \addplot[
      only marks,
      mark=*,
      mark size=2.5pt,
      blue
    ]
    coordinates {(2.8,-0.76)}
    node[above right] {$B$};
  \end{axis}
\end{tikzpicture}
\end{center}
\caption{Illustration of Proposition \ref{PRO:Guan}. We have uniform ellipticity at $A$ and Guan's inequality holds at $B$ if $\delta$ is small enough.}
\end{figure}
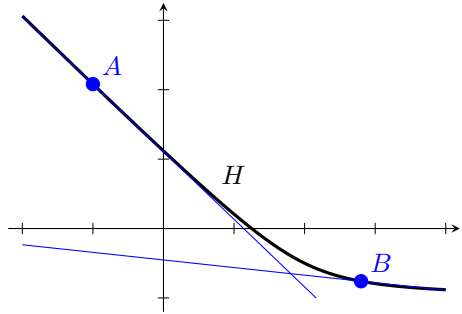
\begin{proof}
Assume without loss of generality that
\begin{equation}\label{EQ:Contradiction inégalité 1}
    -\scal{\nabla F(A)}{A - A_0} \leq \frac{\delta}{2}\scal{\nabla F(A)}{I},
\end{equation}
or else the first inequality holds with $\kappa = \frac{\delta}{2}$.

The assumption that $A \in H^-$ (and $F$ is $\Gamma$-elliptic) implies that $F(A) \leq 0$ so
\begin{align}
    F\!\left(A_0 - \frac{\delta}{2}v\right) & \leq F(A) + \scal{\nabla F(A)}{A_0 - \frac{\delta}{2}I - A} \textrm{ by concavity of $F$}, \nonumber\\
    & \leq -\scal{\nabla F(A)}{A - A_0} - \frac{\delta}{2}\scal{\nabla F(A)}{I} \textrm{ because } F(A) \leq 0, \nonumber\\
    & \leq 0 \textrm{ by \eqref{EQ:Contradiction inégalité 1}}. \label{EQ:A_0 - delta/2 I dans H^-}
\end{align}
We deduce that $A_0 - \frac{\delta}{2}I \in H^-$. In particular, $A_0 - \delta I \in H^-\backslash H$.

Let $P \in \Gamma$ with $\norme{P} = 1$ and let
$$
B = A_0 - \delta I - F_P(A_0 - \delta I)P \in H.
$$
Since $A_0 - \delta I \in H^-\backslash H$, we must have $F_P(A_0 - \delta I) < 0$.

By concavity of $F$ and because $A \in H^-$,
\begin{equation}\label{EQ:Inégalité gradient}
    0 = F(B) \leq F(A) + \scal{\nabla F(A)}{B - A} \leq \scal{\nabla F(A)}{B - A}.
\end{equation}

Therefore,
\begin{align*}
    \scal{\nabla F(A)}{P} & = \scal{\nabla F(A)}{\frac{B - A_0 + \delta I}{-F_P(A_0 - \delta I)}}\\
    & \geq \scal{\nabla F(A)}{\frac{A - A_0 + \delta I}{-F_P(A_0 - \delta I)}} \textrm{ by \eqref{EQ:Inégalité gradient}},\\
    & \geq \scal{\nabla F(A)}{\frac{\frac{\delta}{2}I}{-F_P(A_0 - \delta I)}} \textrm{ by \eqref{EQ:Contradiction inégalité 1}},\\
    & \geq \frac{\delta}{2R}\scal{\nabla F(A)}{I} \textrm{ by \eqref{EQ:Définition R}}.
\end{align*}

For the last inequality, choose $P = I$ in the definition of $B$. We have in any case
\begin{align*}
    -\scal{\nabla F(A)}{A - A_0} & \geq \scal{\nabla F(A)}{A_0 - B} \textrm{ by \eqref{EQ:Inégalité gradient}},\\
    & = \scal{\nabla F(A)}{\delta I + F_I(A_0 - \delta I)I}\\
    & \geq -R\scal{\nabla F(A)}{I} \textrm{ by \eqref{EQ:Définition delta et R}}.
\end{align*}
\end{proof}

\subsection{\texorpdfstring{Families of $\bGamma$}{Gamma}-elliptic convex hypersurfaces}\label{SEC:Famille}

Let $\Gamma$ be a non-empty open cone with strictly convex closure of $E = \R^r$ and $I_0 \in \Gamma$ a point. We endow $E$ with its canonical scalar product.

Let $(X,g)$ be a smooth Riemannian manifold with or without boundary endowed with a real Riemannian vector bundle $(\bE,h) \rightarrow X$ and set $r = \rk(\bE)$. Let $\bGamma \subset \bE$ an open sub-set and $I : X \rightarrow \bE$ a section of $\bE$. Given $U \subset X$ a sub-set, we denote by $\bGamma|_U = \bGamma \cap \bE|_U$ the restriction and when $x \in X$, we write $\Gamma_x = \bGamma|_{\{x\}} \cap E_x$.

Assume that for all $x \in X$, there is a local trivialisation $(e_1,\ldots,e_r)$ of $\bE$ in a neighbourhood $U$ of $x$ in $X$ which is orthonormal at $x$ such that the following diagram commutes.
\begin{equation}\label{EQ:Diagramme commutatif cone}
\begin{tikzcd}
    U \ar[equal]{r}\ar["I|_U"']{d} & U \ar["(\Id_U{,}I_0)"]{d}\\
    \bGamma|_U \ar{r}\ar[hook]{d} & U \times \Gamma \ar[hook]{d}\\
    \bE|_U \ar{r}\ar{d} & U \times E \ar{d}\\
    U \ar[equal]{r} & U
\end{tikzcd}
\end{equation}
The horizontal arrows are given by the trivialisation $(e_1,\ldots,e_r)$.

\begin{example}
    We give three examples of tuples $(E,\Gamma_0,I_0,\bE,\bGamma,I)$ that satisfy this property.
    \begin{enumerate}
        \item $E = \R^n$, $\Gamma_0 = (\R_+^*)^n$, $I_0 = (1,\ldots,1)$, $\bE = X \times \R^n$ is the trivial bundle, $\bGamma = X \times \Gamma_0$, $I : x \mapsto (x,I_0)$.
        \item $E = \End_S(\R,n)$ is the space of $n \times n$ symmetric matrices, $\Gamma_0 = \End_S^{++}(\R,n)$ the cone of positive definite matrices, $I_0 = \Id_n$ the identity, $\bE = \End_S(TX,g)$ the bundle of $g$-symmetric endomorphisms of the tangent bundle of $X$, $\bGamma = \End_S^{++}(TX,g)$ the open sub-set of positive definite endomorphisms and $I = \Id$ the identity section. We could also replace $(TX,g)$ by another bundle equipped with a metric.
        \item Assume in this example that $(X,g)$ is a complex Hermitian manifold of complex dimension $n$. $E = \End_H(\C,n)$ is the space of $n \times n$ complex Hermitian matrices, $\Gamma_0 = \End_H^{++}(\C,n)$ the cone of positive definite matrices, $I_0 = \Id_n$ the identity, $\bE = \End_H(TX,g)$ the bundle of $g$-Hermitian endomorphisms of the tangent bundle of $X$, $\bGamma = \End_H^{++}(TX,g)$ the open sub-set of positive definite endomorphisms and $I = \Id$ the identity section. We could also replace $(TX,g)$ by another complex bundle equipped with a Hermitian metric (in this case, we don't need to assume that $X$ admits a complex structure).
    \end{enumerate}
\end{example}
We will mostly be interested here in the second example.

\begin{definition}\label{DEF:Famille}
    We say that $\bH \subset \bE$ is a \textit{family of smooth $\bGamma$-elliptic convex hypersurfaces} if it is a smooth hypersurface of $\bE$ which is transverse to each fibre and for all $x \in X$, $H_x$ is a $\Gamma_x$-elliptic convex hypersurface of $(E_x,h)$.
\end{definition}

For each $x \in X$, we denote by $F_I(x,\cdot) : E_x \rightarrow \R$ the function defined in Sub-section \ref{SEC:Fonction F_P} associated with $H_x$ and $I(x)$. It defines a global function
$$
F_I : \bE \longrightarrow \R
$$
that depends on $\bH$ and the section $I$. By Proposition \ref{PRO:F_I bien définie}, it is entirely characterised by the property that for all $(x,A) \in \bE$,
$$
A - F_I(x,A)I(x) \in H_x.
$$

Given $(x,A) \in \bE$, we have a natural decomposition
\begin{equation}\label{EQ:Décomposition orthogonale TE}
    T_{(x,A)}\bE = T_xX \oplus E_x.
\end{equation}
Moreover, this decomposition is orthogonal with respect to the natural metric on $\bE$ induced by $g$ and $h$. We denote by $\pi_{T_xX}$ and $\pi_{E_x}$ the projections in \eqref{EQ:Décomposition orthogonale TE}.

Assume now that $(x,A) \in \bH$. Since $H_x$ is a $\Gamma_x$-elliptic convex hypersurface of $E_x$, there is a canonical choice of unit normal vector $\NN_{H_x}(A) \in \Gamma_x^\lor$ at $A \in H_x$. We choose $\NN_\bH(x,A)$ to be the unit normal vector at $(x,A) \in \bH$ such that the projection $\pi_{E_x}\NN_\bH(x,A)$ is a positive multiple of $\NN_{H_x}(A)$. Notice that since we assumed $\bH$ to be transverse to $E_x$, this projection is non-zero.

Concretely, we have
\begin{equation}
    \NN_{H_x}(A) = \frac{\pi_{E_x}\NN_\bH(x,A)}{\norme{\pi_{E_x}\NN_\bH(x,A)}}.
\end{equation}
The norm on $E_x$ is induced by $h$.

Since we have a canonical choice of unit normal vector, we can define the associated second fundamental form.
$$
\2_\bH(x,A)((\hat{x},\hat{A}),(\hat{x},\hat{A})) = -\scal{d\NN_\bH(x,A)(\hat{x},\hat{A})}{(\hat{x},\hat{A})}
$$
Here, $(\hat{x},\hat{A}) \in T_{(x,A)}\bH$ and we use the orthogonal decomposition \eqref{EQ:Décomposition orthogonale TE}. Notice that $\hat{A} \in T_AH_x$.

\begin{proposition}\label{PRO:Différentielles F_I}
    $F_I$ is smooth and its gradient and second differential at each point $(x,A) \in \bE$ are given by
    $$
    \nabla F_I(x,A) = \frac{\NN_{\bH}(x,A - F_I(x,A)I(x))}{\scal{\NN_{\bH}(x,A - F_I(x,A)I(x))}{I(x)}},
    $$
    \begin{align*}
        & \quad\ d^2F_I(x,A)((\hat{x},\hat{A}),(\hat{x},\hat{A}))\\
        & = -\frac{\2_\bH(x,A - F_I(x,A)I(x))\Big((\hat{x},\hat{A} - (dF_I(x,A)(\hat{x},\hat{A}))I(x)),(\hat{x},\hat{A} - (dF_I(x,A)(\hat{x},\hat{A}))I(x))\Big)}{\scal{\NN_{\bH}(x,A - F_I(x,A)I(x))}{I(x)}}.
    \end{align*}
\end{proposition}
\begin{proof}
By working in local coordinates, we may assume that $\bE = U \times E$, $\bGamma = U \times \Gamma$ and $I = I_0$ is constant. The proof for the first equality is the exact same as the one of Lemma \ref{LEM:TFI construction F_P} and the proof for the second equality follows from the same computation as \eqref{EQ:Hessienne F_I}.
\end{proof}

\subsection{Assumption on \texorpdfstring{$\bH$}{IH}}\label{SEC:Hypothèse}

From now on, we assume that $X$ is closed and $E = \End_S(TX,g)$ is the fibre bundle of $g$-symmetric endomorphisms of the tangent bundle $TX$. We denote by $\End_S^{++}(TX,g)$ the set of definite positive endomorphisms and $\Id$ the identity section. Notice that $\End_S(TX,g)$, $\End_S^{++}(TX,g)$ and $\Id$ satisfy \eqref{EQ:Diagramme commutatif cone}. We consider $\bH \subset E$ a family of $\End_S^{++}(TX,g)$-elliptic convex hypersurfaces and we want to obtain \textit{a priori} estimates on smooth solutions of \eqref{EQ:Equation}. Let $F = F_\Id$ so $\bH$ is the zero locus of $F$.

To make the analysis work in Section \ref{SEC:Estimées}, we need additional assumptions on $\bH$. Let $\mathfrak{S}_n$ be the group of permutations of $n$ elements that acts on $\R^n$ by permuting the coordinates. For the next theorem, for any Euclidean space $E$ of dimension $n$, we set
$$
\lambda : \End_S(E) \longrightarrow \R^n/\mathfrak{S}_n,
$$
the continuous map that associates to each symmetric endomorphism its eigenvalues. Let $\underline{M} = (M_1,M_2,M_3,M_4,M_5,M_6,M_7,M_8,M_9)$ be a tuple of positive constants. We make the following technical assumption on $\bH$.

\begin{definition}\label{DEF:Presque spectrale}
    We say that $\bH$ is \textit{almost spectral} with tuple $\underline{M}$ if it satisfies the following property.
    
    For all $(x_0,A_0) \in \bH$ with $\norme{A_0} \geq M_1$, we can write locally $\bH$ around $(x_0,A_0)$ as
    $$
    \bH = \{(x,A) \in \End_S(TX,g) \mid f(\lambda(B(x) + A)) + G(x,A) = 0\},
    $$
    where
    \begin{enumerate}
        \item $(\R_+^*)^n \subset \Gamma \subset \R^n$ is an $\mathfrak{S}_n$-invariant open cone and
        $$
        f : \Gamma \longrightarrow \R,
        $$
        is a smooth $\mathfrak{S}_n$-invariant symmetric $\R_+^n$-elliptic concave function.
        \item $B : (X,x_0) \longrightarrow \End_S(TX,g)$ is a smooth germ of sections such that $\lambda(B(x_0) + A_0) \in \Gamma$.
        \item We have
        $$
        \lim_{t \rightarrow 0^+} f(t\lambda(B(x_0) + A_0)) < f(\lambda(B(x_0) + A_0)) < \lim_{t \rightarrow +\infty} f(t\lambda(B(x_0) + A_0)).
        $$
        \item $G : (\End_S(TX,g),(x_0,A_0)) \longrightarrow \R$ is a smooth germ of functions such that $G(x_0,A_0) = -f(\lambda(B(x_0) + A_0))$.
        \item Let $\cF = \scal{\nabla f(\lambda(B(x_0) + A_0))}{\lambda(\Id)} = \sum_{k = 1}^n \frac{\partial f}{\partial\lambda_k}(\lambda(B(x_0) + A_0))$. We have the bounds
        \begin{equation}\label{EQ:Hypothèse bornes B}
            \norme{B(x_0)} \leq M_2, \qquad \norme{\nabla B(x_0)} \leq M_3,
        \end{equation}
        \begin{equation}\label{EQ:Hypothèse nabla G}
            \norme{\nabla_xG(x_0,A_0)} \leq M_4\cF, \qquad \norme{\nabla_AG(x_0,A_0)} \leq \frac{M_5}{\norme{A_0}}\cF,
        \end{equation}
        \begin{equation}\label{EQ:Hypothèse encadrement nabla_AG}
            -\left(1 - \frac{1}{M_6}\right)\nabla(f \circ \lambda)(B(x_0) + A_0) \leq \nabla_AG(x_0,A_0) \leq M_7\nabla(f \circ \lambda)(B(x_0) + A_0),
        \end{equation}
        \item For all $(\hat{x},\hat{A}) \in T_{(x_0,A_0)}\bH$,
        \begin{equation}\label{EQ:Hypothèse nabla2G}
            -\nabla^2G(x_0,A_0)((\hat{x},\hat{A}),(\hat{x},\hat{A})) \geq -M_8\cF\norme{A_0}\norme{\hat{x}}^2 + \frac{M_9}{2}d^2(f \circ \lambda)(B(x_0) + A_0)(\hat{A},\hat{A}).
        \end{equation}
        \item $M_7 + M_9 \leq 1$.
    \end{enumerate}
\end{definition}

\begin{remark}
    When we can choose $f$ and $B$ globally and independently of the point $(x_0,A_0)$ and when $G = 0$, we recover the standard frame of Caffarelli, Nirenberg and Spruck \cite{CNS}, for which the \textit{a priori} estimates are well-known \cite{Li}\cite{Urbas}\cite{Guan}\cite{Szekelyhidi}. Roughly speaking, Definition \ref{DEF:Presque spectrale} requires that the non-spectral part represented by $G$ is controlled by the spectral part.
\end{remark}

Let $(x_1,\ldots,x_n)$ be a local orthonormal frame around a point $x_0 \in X$. It induces a local trivialisation of $\End_S(TX,g)$ at $x_0$ given by the $A_{ii} = \frac{\partial}{\partial x_i} \otimes dx_i$ and the $A_{ij} = A_{ji} = \frac{\sqrt{2}}{2}\left(\frac{\partial}{\partial x_i} \otimes dx_j + \frac{\partial}{\partial x_j} \otimes dx_i\right)$ when $j > i$. This trivialisation is orthonormal at $x_0$ too and the identity section has constant coefficients in it.

We use the following notations for the covariant derivatives with respect to the Levi--Civita connection
$$
u_{i_1i_2\cdots i_r} = \nabla_{\frac{\partial}{\partial x_{i_r}}}\cdots\nabla_{\frac{\partial}{\partial x_{i_2}}}\nabla_{\frac{\partial}{\partial x_{i_1}}}u.
$$
Similarly, for $F : \End_S(TX,g) \rightarrow \R$,
$$
F_{k,} = \nabla_{\frac{\partial}{\partial x_k}}F, \qquad F_{,ij} = \nabla_{\frac{\partial}{\partial A_{ij}}}F.
$$
And for a real function $f : \R^n \rightarrow \R$, we also denote by
$$
f_{i_1i_2\cdots i_r} = \frac{\partial^rf}{\partial\lambda_{i_1}\partial\lambda_{i_2}\cdots\partial\lambda_{i_r}}.
$$

Let $U \subset \R^n$ a $\mathfrak{S}_n$-invariant open sub-set and $f : U \rightarrow \R$ a smooth $\mathfrak{S}_n$-invariant function. If $E$ is a Euclidean space of dimension $n$ and $A \in \End_S(E)$ a point that is diagonal with pairwise distinct eigenvalues $(\lambda_1,\ldots,\lambda_n)$ in an orthonormal basis say $(\hat{x}_1,\ldots,\hat{x}_n)$, we have for all indices $i,j,k,l$,
\begin{equation}\label{EQ:Dérivée première f o lambda}
    \frac{\partial(f \circ \lambda)}{\partial A_{ij}}(A) = \delta_{ij}f_i(\lambda(A))
\end{equation}
\begin{equation}\label{EQ:Dérivée seconde f o lambda}
    \frac{\partial^2(f \circ \lambda)}{\partial A_{ij}\partial A_{kl}}(A) = \delta_{ij}\delta_{kl}f_{ik}(\lambda(A)) + (1 - \delta_{ij})\delta_{il}\delta_{jk}\frac{f_i(\lambda(A)) - f_j(\lambda(A))}{\lambda_i - \lambda_j}
\end{equation}
See for example \cite{Spruck}.

\begin{lemma}\label{LEM:Conséquences hypothèse}
    If $\bH$ is almost spectral with tuple $\underline{M}$, there is a tuple $\underline{M}' = (M_{10},M_{11},M_{12},M_{13},M_{14},M_{15})$ that only depends on $\underline{M}$ and $n$ such that for all $(x_0,A_0) \in \bH$ with $\norme{A_0} \geq M_{10}$ which is diagonal in the basis $\frac{\partial}{\partial x_1}(x_0),\ldots,\frac{\partial}{\partial x_n}(x_0)$ with eigenvalues $\lambda_1 \geq \ldots \geq \lambda_n$ and $\lambda_1 > \lambda_n$,
    \begin{enumerate}
        \item $F_{,nn}(x_0,A_0) \geq \frac{1}{M_{11}}$,
        \item $\sum_{i = 1}^n F_{,ii}(x_0,A_0)\lambda_i \geq -M_{12}$,
        \item For all $k$, $\abs{F_{k,}(x_0,A_0)} \leq M_{13}$,
        \item For all $(\hat{x},\hat{A}) \in T_{(x_0,A_0)}\bH$,
        $$
        -d^2F(x_0,A_0)((\hat{x},\hat{A}),(\hat{x},\hat{A})) \geq -M_{14}\norme{A_0} + \frac{1}{\lambda_1 - \lambda_n}\sum_{i,j = 1}^n (F_{,ij}(x_0,A_0) - M_{15}\delta_{ij}F_{,11}(x_0,A_0))\hat{A}_{1i}\hat{A}_{1j}.
        $$
    \end{enumerate}
\end{lemma}
\begin{proof}
By a density argument, we may assume without loss of generality that $B(x_0) + A_0$ has pairwise distinct eigenvalues $\lambda_1 < \ldots < \lambda_n$.

Let $M_{10} = \max\{M_1,2nM_5\}$. In particular, if $\norme{A_0} \geq M_{10}$, then $\norme{A_0} \geq M_1$ so the properties in Definition \ref{DEF:Presque spectrale} apply. Let $f$, $B$ and $G$ given by this definition. Recall that $\cF = \sum_{i = 1}^n \frac{\partial f}{\partial\lambda_i}(\lambda(B(x_0) + A_0))$ is positive by ellipticity of $f$. Since $\bH$ is locally given defined by the equation $f(\lambda(B(x) + A)) + G(x,A) = 0$, we have, for all $(\hat{x},\hat{A}) \in T_{(x_0,A_0)}\End_S(T_{x_0}X,g)$,
\begin{equation}\label{EQ:Vecteur normal en fonction de f et G}
    L\scal{\NN_\bH(x_0,A_0)}{(\hat{x},\hat{A})} = d(f \circ \lambda)(dB(x_0)\hat{x} + \hat{A}) + dG(x_0,A_0)(\hat{x},\hat{A}).
\end{equation}
If moreover, $(\hat{x},\hat{A}) \in T_{(x_0,A_0)}\bH$,
\begin{align}\label{EQ:II en fonction de f et G}
\begin{split}
    L\,\2_\bH(x_0,A_0)((\hat{x},\hat{A}),(\hat{x},\hat{A})) & = -d^2(f \circ \lambda)(x_0,B_0 + A_0)((dB(x_0)\hat{x} + \hat{A}),(dB(x_0)\hat{x} + \hat{A}))\\
    & - d(f \circ \lambda)(x_0,B_0 + A_0)d^2B(x_0)(\hat{x},\hat{x}) - d^2G(x_0,A_0)((\hat{x},\hat{A}),(\hat{x},\hat{A})).
\end{split}
\end{align}
where $L \in \R$ is a normalisation constant. Similarly, by Proposition \ref{PRO:Différentielles F_I}, for all $(\hat{x},\hat{A})$,
\begin{equation}\label{EQ:Vecteur normal en fonction de F}
    \frac{\scal{\NN_\bH(x_0,A_0)}{(\hat{x},\hat{A})}}{\scal{\NN_\bH(x_0,A_0)}{(0,\Id)}} = dF(x_0,A_0)(\hat{x},\hat{A}),
\end{equation}
and if $(\hat{x},\hat{A})$ is tangent to $\bH$,
\begin{align}\label{EQ:II en fonction de F}
    \frac{\2_\bH(x_0,A_0)((\hat{x},\hat{A}),(\hat{x},\hat{A}))}{\scal{\NN_\bH(x_0,A_0)}{(0,\Id)}} = -d^2F(x_0,A_0)((\hat{x},\hat{A}),(\hat{x},\hat{A})).
\end{align}

\noindent\textbf{Step 1:} Lower bound on $L$.

From now on, we always evaluate functions at $(x_0,A_0)$ so we omit to write it. Assume without loss of generality that $\norme{A_0} \geq M_{10} \geq 2nM_5$. In particular,
\begin{equation}\label{EQ:Conséquence A_0 grande}
    \frac{M_5}{\norme{A_0}} \leq \frac{1}{2n} \leq \frac{1}{2\sqrt{n}}.
\end{equation}

Applying \eqref{EQ:Vecteur normal en fonction de f et G} with $\hat{x} = 0$ and $\hat{A} = \Id$, we obtain that
\begin{align}
    L\scal{\NN_\bH}{(0,\Id)} & = \cF + d_AG\,\Id \geq \cF - \norme{\nabla_AG}\norme{\Id} \nonumber\\
    & \geq \left(1 - \frac{M_5\sqrt{n}}{\norme{A_0}}\right)\cF \textrm{ by \eqref{EQ:Hypothèse nabla G}} \nonumber\\
    & \geq \frac{1}{2}\cF \textrm{ by \eqref{EQ:Conséquence A_0 grande}}. \label{EQ:L fois PS > F/2}.
\end{align}
By positivity of $\NN_\bH$, $\scal{\NN_\bH}{(0,\Id)} \geq \norme{\NN_\bH} = 1$. We deduce that $L$ is positive and
\begin{equation}\label{EQ:Borne inférieure L}
    L \geq \frac{1}{2}\cF > 0.
\end{equation}

\noindent\textbf{Step 2:} Proof of point 3.

We have
\begin{align}
    L\norme{\pi_{T_{x_0}X}\NN_\bH}^2 & = L\scal{\NN_\bH}{(\pi_{T_{x_0}X}\NN_\bH,0)} \nonumber\\
    & = d(f \circ \lambda)dB\,(\pi_{T_{x_0}X}\NN_\bH) + d_xG\,(\pi_{T_{x_0}X}\NN_\bH) \textrm{ by \eqref{EQ:Vecteur normal en fonction de f et G}}, \nonumber\\
    & \leq \norme{\nabla(f \circ \lambda)}\norme{\nabla B}\norme{\pi_{T_{x_0}X}\NN_\bH} + M_4\cF\norme{\pi_{T_{x_0}X}\NN_\bH} \textrm{ by \eqref{EQ:Hypothèse nabla G}}, \nonumber\\
    & \leq (M_3 + M_4)\cF\norme{\pi_{T_{x_0}X}\NN_\bH} \textrm{ by ellipticity of $f$ and \eqref{EQ:Hypothèse bornes B}}, \nonumber\\
    L\norme{\pi_{T_{x_0}X}\NN_\bH} & \leq (M_3 + M_4)\cF. \label{EQ:L norme x}
\end{align}

Therefore, by \eqref{EQ:Vecteur normal en fonction de F}, \eqref{EQ:L fois PS > F/2} and \eqref{EQ:L norme x},
$$
\norme{\nabla_xF} = \frac{\norme{\pi_{T_{x_0}X}\NN_\bH}}{\scal{\NN_\bH}{(0,\Id)}} \leq 2(M_3 + M_4).
$$

\noindent\textbf{Step 3:} Upper bound on $L$.

On the one hand, we have
\begin{align*}
    L\scal{\NN_\bH}{(0,\Id)} & = \cF + d_AG\,\Id \textrm{ by \eqref{EQ:Vecteur normal en fonction de f et G}},\\
    & \leq \cF + \norme{\nabla_AG}\norme{\Id}\\
    & \leq \left(1 + \frac{M_5\sqrt{n}}{\norme{A_0}}\right)\cF \textrm{ by \eqref{EQ:Hypothèse nabla G}},\\
    & \leq \frac{3}{2}\cF \textrm{ by \eqref{EQ:Conséquence A_0 grande}},
\end{align*}
and on the other hand,
\begin{align}
    \scal{\NN_\bH}{(0,\Id)} & \geq \norme{\pi_{\End_S(T_{x_0}X,g)}\NN_\bH} \textrm{ by ellipticity}, \nonumber\\
    & \geq 1 - \norme{\pi_{T_{x_0}X}\NN_\bH} \nonumber\\
    & = 1 - \norme{\nabla_xF}\scal{\NN_\bH}{(0,\Id)} \textrm{ by \eqref{EQ:Vecteur normal en fonction de F}}, \nonumber\\
    & \geq 1 - 2(M_3 + M_4)\scal{\NN_\bH}{(0,\Id)}. \label{EQ:Borne inférieure NN contre Id}
\end{align}
Hence,
\begin{equation}\label{EQ:Borne supérieure L}
    L \leq \frac{3\cF}{2\scal{\NN_\bH}{(0,\Id)}} \leq \frac{3(2(M_3 + M_4) + 1)}{2}\cF.
\end{equation}

\noindent\textbf{Step 4:} Proof of point 1.

Recall that we work in an orthonormal basis $\frac{\partial}{\partial x_1},\ldots,\frac{\partial}{\partial x_n}$ where $B(x_0) + A_0$ is diagonal with pairwise distinct eigenvalues. By \eqref{EQ:Dérivée première f o lambda}, in this basis,
$$
\nabla(f \circ \lambda) =
\begin{pmatrix}
    f_1 & 0 & \hdots & 0\\
    0 & f_2 & \hdots & 0\\
    \vdots & \vdots & \ddots & \vdots\\
    0 & 0 & \hdots & f_n
\end{pmatrix}.
$$
Moreover, it has decreasing eigenvalues,
$$
f_1 \leq \cdots \leq f_n.
$$

In particular,
\begin{equation}\label{EQ:Borne inférieure Fnn}
    \hat{x}_n\nabla(f \circ \lambda)\hat{x}_n = f_n \geq \frac{1}{n}\cF.
\end{equation}
We deduce that
\begin{align*}
    F_{,nn} & = \hat{x}_n^*\frac{\pi_{\End_S(T_{x_0}X,g)}\NN_\bH}{\scal{\NN_\bH}{(0,\Id)}}\hat{x}_n \textrm{ by \eqref{EQ:Vecteur normal en fonction de F}},\\
    & \geq \hat{x}_n^*\frac{\pi_{\End_S(T_{x_0}X,g)}\NN_\bH}{\norme{\pi_{\End_S(T_{x_0}X,g)}\NN_\bH}\norme{\Id}}\hat{x}_n\\
    & \geq \frac{1}{\sqrt{n}}\hat{x}_n^*(\pi_{\End_S(T_{x_0}X,g)}\NN_\bH)\hat{x}_n\\
    & = \frac{1}{\sqrt{n}L}(\hat{x}_n^*\nabla(f \circ \lambda)\hat{x}_n + \hat{x}_n^*\nabla_AG\hat{x}_n) \textrm{ by \eqref{EQ:Vecteur normal en fonction de f et G}},\\
    & \geq \frac{1}{\sqrt{n}L}\left(\frac{1}{n}\cF - \frac{M_5}{\norme{A}_0}\cF\right) \textrm{ by \eqref{EQ:Borne inférieure Fnn} and \eqref{EQ:Hypothèse nabla G}},\\
    & \geq \frac{1}{2n^{3/2}L}\cF \textrm{ by \eqref{EQ:Conséquence A_0 grande}},\\
    & \geq \frac{1}{3n^{3/2}(2(M_3 + M_4) + 1)} \textrm{ by \eqref{EQ:Borne supérieure L}}.
\end{align*}

\textbf{Step 5:} Proof of point 2.

Let the real function
$$
a : \fonction{\R_+^*}{\R_+}{t}{f(t\lambda(B(x_0) + A_0))}.
$$
By assumption on $f$, $\lim_{t \rightarrow 0} a(t) < a(1) < \lim_{t \rightarrow +\infty}$ and $a$ is concave. In particular, $a$ must be non-decreasing near $1$ so $a'(1) \geq 0$ \textit{i.e.}
\begin{equation}\label{EQ:Non Guan}
    \scal{\nabla(f \circ \lambda)}{B(x_0) + A_0} \geq 0.
\end{equation}

Therefore,
\begin{align*}
    \sum_{i = 1}^n F_{,ii}\lambda_i & = \scal{\nabla_AF}{A_0}\\
    & = \frac{\scal{\NN_\bH}{(0,A_0)}}{\scal{\NN_\bH}{(0,\Id)}} \textrm{ by \eqref{EQ:Vecteur normal en fonction de F}},\\
    & = \frac{1}{L\scal{\NN_\bH}{(0,\Id)}}(\scal{\nabla(f \circ \lambda)}{A_0} + \scal{\nabla_AG}{A_0}) \textrm{ by \eqref{EQ:Vecteur normal en fonction de f et G}},\\
    & \geq \frac{1}{L\scal{\NN_\bH}{(0,\Id)}}(-\scal{\nabla(f \circ \lambda)}{B(x_0)} - \norme{\nabla_AG}\norme{A_0}) \textrm{ by \eqref{EQ:Non Guan}},\\
    & \geq \frac{1}{L\scal{\NN_\bH}{(0,\Id)}}\left(-\norme{B(x_0)} - M_5\right)\cF \textrm{ by \eqref{EQ:Hypothèse nabla G}},\\
    & \geq -2(M_2 + M_5) \textrm{ by \eqref{EQ:Hypothèse bornes B} and \eqref{EQ:Borne inférieure L}}.
\end{align*}

\noindent\textbf{Step 6:} Proof of point 4.

Let $(\hat{x},\hat{A}) \in T_{(x_0,A_0)}\bH$. By \eqref{EQ:Dérivée seconde f o lambda},
\begin{align}
    -d^2(f \circ \lambda)(\hat{A},\hat{A}) & = -\sum_{i,j = 1}^n f_{ij}\hat{A}_{ii}\hat{A}_{jj} - \sum_{\underset{i \neq j}{i,j = 1}}^n \frac{f_i - f_j}{\lambda_i - \lambda_j}\hat{A}_{ij}^2 \nonumber\\
    & \geq -\sum_{\underset{i \neq j}{i,j = 1}}^n \frac{f_i - f_j}{\lambda_i - \lambda_j}\hat{A}_{ij}^2 \textrm{ by concavity of $f$}. \nonumber\\
    & \geq 2\sum_{i = 2}^n \frac{f_i - f_1}{\lambda_1 - \lambda_i}\hat{A}_{1i}^2 \nonumber\\
    & \geq \frac{2}{\lambda_1 - \lambda_n}\sum_{i = 1}^n (f_i - f_1)\hat{A}_{1i}^2. \label{EQ:Inégalité concavité 1}
\end{align}

Let for all $i$, $v_i = A_{1i}$. We have
\begin{align}
    \sum_{i,j = 1}^n F_{,ij}\hat{A}_{1i}\hat{Aj} & = v^*\nabla_AF\,v \nonumber\\
    & = v^*\frac{\pi_{\End_S(T_{x_0}X,g)}\NN_\bH}{\scal{\NN_\bH}{(0,\Id)}}v \textrm{ by \eqref{EQ:Vecteur normal en fonction de F}}, \nonumber\\
    & = v^*\frac{\nabla(f \circ \lambda) + \nabla_AG}{L\scal{\NN_\bH}{(0,\Id)}}v \textrm{ by \eqref{EQ:Vecteur normal en fonction de f et G}}, \nonumber\\
    & \leq (1 + M_7)v^*\frac{\nabla(f \circ \lambda)}{L\scal{\NN_\bH}{(0,\Id)}}v \textrm{ by \eqref{EQ:Hypothèse encadrement nabla_AG}}, \nonumber\\
    & = \frac{1 + M_7}{L\scal{\NN_\bH}{(0,\Id)}}\sum_{i = 1}^n f_i\hat{A}_{1i}^2, \label{EQ:Inégalité concavité 2}
\end{align}
and with the same computation,
\begin{align}
    \sum_{i,j = 1}^n F_{,11}\hat{A}_{1i}\hat{Aj} & = \hat{x}_1^*\nabla_A F\,\hat{x}_1\norme{v}^2 \nonumber\\
    & \geq \frac{1}{M_6}\frac{1}{L\scal{\NN_\bH}{(0,\Id)}}\sum_{i = 1}^n f_1\hat{A}_{1i}^2 \textrm{ by \eqref{EQ:Hypothèse encadrement nabla_AG}}. \label{EQ:Inégalité concavité 3}
\end{align}

Finally,
\begin{adjustwidth}{-5cm}{5cm}
\begin{align*}
    -d^2F((\hat{x},\hat{A}),(\hat{x},\hat{A})) & = \frac{1}{L\scal{\NN_\bH}{(0,\Id)}}(-d^2(f \circ \lambda)(\hat{A},\hat{A}) - d^2G((\hat{x},\hat{A}),(\hat{x},\hat{A}))) \textrm{ by \eqref{EQ:II en fonction de f et G} and \eqref{EQ:II en fonction de F}},\\
    & \geq \frac{1}{L\scal{\NN_\bH}{(0,\Id)}}\left(-d^2(f \circ \lambda)(\hat{A},\hat{A}) - M_8\cF\norme{A_0}\norme{\hat{x}}^2 + \frac{M_9}{2}d^2(f \circ \lambda)((\hat{x},\hat{A}),(\hat{x},\hat{A}))\right) \textrm{ by \eqref{EQ:Hypothèse nabla2G}},\\
    & = -\frac{M_8\cF}{L\scal{\NN_\bH}{(0,\Id)}}\norme{A_0}\norme{x}^2 - \frac{1 - \frac{M_9}{2}}{L\scal{\NN_\bH}{(0,\Id)}}d^2(f \circ \lambda)(\hat{A},\hat{A})\\
    & \geq -2(2(M_3 + M_4) + 1)M_8\norme{A_0} + \frac{2 - M_9}{L\scal{\NN_\bH}{(0,\Id)}}\frac{1}{\lambda_1 - \lambda_n}\sum_{i = 1}^n (f_i - f_1)\hat{A}_{1i}^2\\
    & \textrm{ by \eqref{EQ:Borne inférieure L}, \eqref{EQ:Borne inférieure NN contre Id} and \eqref{EQ:Inégalité concavité 1}}\\
    & \geq -2(2(M_3 + M_4) + 1)M_8\norme{A_0} + \frac{2 - M_9}{1 + M_7}\frac{1}{\lambda_1 - \lambda_n}\sum_{i,j = 1}^n F_{,ij}\hat{A}_{1i}\hat{A}_{1j}\\
    & - M_6\left(2 - M_9\right)\frac{1}{\lambda_1 - \lambda_n}\sum_{i = 1}^n F_{,11}\hat{A}_{1i}^2 \textrm{ by \eqref{EQ:Inégalité concavité 2} and \eqref{EQ:Inégalité concavité 3}},\\
    & \geq -2(2(M_3 + M_4) + 1)M_8\norme{A_0} + \frac{1}{\lambda_1 - \lambda_n}\sum_{i,j = 1}^n F_{,ij}\hat{A}_{1i}\hat{A}_{1j}\\
    & - M_6\left(2 - M_9\right)\frac{1}{\lambda_1 - \lambda_n}\sum_{i = 1}^n F_{,11}\hat{A}_{1i}^2 \textrm{ by point 7}.
\end{align*}
\end{adjustwidth}
\end{proof}

\section{\textit{A priori estimates}}\label{SEC:Estimées}

\subsection{\texorpdfstring{$L^1$}{L1} and \texorpdfstring{$\cC^0$}{C0} estimates}\label{SEC:Estimées C0}

Let $\bH \subset \End_S(TX,g)$ a family of smooth $\End_S^{++}(TX,g)$-elliptic convex hypersurfaces of $\End_S(TX,g)$ on $X$. We do not require $\bH$ to be almost spectral in this sub-section.

\begin{definition}\label{DEF:Sous-solution}
    We say that a smooth function $\underline{u} : X \rightarrow \R$ is a \textit{$\cC$-sub-solution} for $\bH$ if for all $x \in X$, $\nabla^2\underline{u}(x) \in \Omega_{H_x}$.
\end{definition}

We assume that $0$ is a $\cC$-\textit{sub-solution}, meaning that for all $x$, $0 \in \Omega_{H_x}$.

In general, if there is any sub-solution $\underline{u} : X \rightarrow \R$, we can translate $\bH$ by $-\nabla^2\underline{u}$ to make $0$ a $\cC$-sub-solution. By Proposition \ref{COR:Propriétés Omega_H}, each $\Omega_{H_x}$ is open and since $X$ is compact, there are positive numbers $\delta$ and $R$ such that
\begin{equation}\label{EQ:Définition delta et R}
    \forall (x,A) \in \End_S^{++}(X,g), -\delta\Id + A \in H_x \quad \Longrightarrow \quad \norme{A} \leq R.
\end{equation}

We start by showing $L^1$ and $\cC^0$ estimates for \eqref{EQ:Equation}.

Usually, solutions of equations of this type (the Monge--Ampère equation for example) have their Laplacian trivially bounded from below. It is not the case here in general so we introduce another linear elliptic operator that will replace the Laplacian.

Let $\bH_0 \subset \End_S(TX,g)$ be a family of smooth $\End_S^{++}(TX,g)$-elliptic convex hypersurfaces of $\End_S(TX,g)$ such that
\begin{enumerate}
    \item Each $H_{0,x} \subset \End_S(T_xX,g)$ is an affine hyperplane.
    \item For all $x$, $H_{0,x} \subset H_x^-$.
\end{enumerate}
Such a family always exists. We can take for example
$$
\bH_0 = \{(x,A) \in \End_S(TX,g) \mid \scal{\NN_{H_x}(-F_\Id(x,0)\Id)}{A} = \scal{\NN_{H_x}(-F_\Id(x,0)\Id)}{-F_\Id(x,0)\Id}\},
$$
which is indeed included in $H_x^-$ by concavity of $F_\Id(x,\cdot)$.

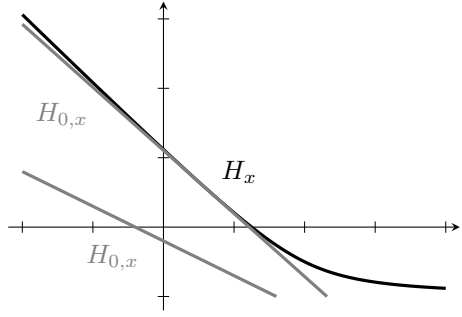
\begin{figure}[H]
\begin{center}
\begin{tikzpicture}
  \begin{axis}[
    axis lines=middle,
    xmin=-2.2, xmax=4.2,
    ymin=-1.2, ymax=3.25,
    width=0.5\linewidth,
    height=0.375\linewidth,
    samples=200,
    domain=-2:4,
    xtick={-2,-1,0,1,2,3,4},
    ytick={-1,0,1,2,3,4},
    xticklabels={},
    yticklabels={},
    tick style={black},
    axis line style={black},
    clip=false
  ]

    \addplot[
      black,
      very thick,
      domain=-2:4
    ]
    {0.5*(-x + sqrt((x-2)^2 + 1))}
    node[midway,above right] {$H_x$};

    %
    \pgfmathsetmacro{\a}{(sqrt(11)-1)/4}
    \pgfmathsetmacro{\m}{-(7+2*sqrt(11))/15}

    \addplot[
      gray,
      very thick,
      domain=-2:2.317
    ]
    {
      \a + \m*(x-\a)
    }
    node[near start,below left] {$H_{0,x}$};

    \addplot[
      gray,
      very thick,
      domain=-2:1.6
    ]
    {
      -x/2 - 0.2
    }
    node[midway,below left] {$H_{0,x}$};

  \end{axis}
\end{tikzpicture}
\end{center}
\caption{Two examples of possible affine hyperplanes $H_{0,x} \subset H_x^-$.}\label{FIG:H_0}
\end{figure}

We set $\NN_0(x) \in \End_S^{++}(T_xX,g)$ the normal vector of $H_{0,x}$ at any point, which depends smoothly on $x$, and $a : X \rightarrow \R$ the smooth function such that
$$
H_{0,x} = \{A \in \End_S(T_xX,g) \mid \scal{\NN_0(x)}{A} = a(x)\}.
$$

The $L^1$ estimate doesn't require the existence of a $\cC$-sub-solution but it depends on $\bH_0$.

\begin{proposition}[$L^1$ estimates]\label{PRO:Estimées L1}
    There is a constant $C = C(X,g,\bH_0)$ such that for all smooth solutions $u : X \rightarrow \R$ of \eqref{EQ:Equation} with $\max_X u = 0$,
    $$
    \norme{u}_{L^1} \leq C.
    $$
\end{proposition}
\begin{proof}
The proof is the same as the one of Tosatti and Weinkove \cite[Lemma 3.3]{Tosatti_Weinkove} except that we do not have a lower bound on the Laplacian of $u$ so we must introduce another non-positive linear elliptic operator that depends on $\bH_0$,
$$
\square u(x) = \scal{\nabla^2u(x)}{\NN_0(x)}.
$$
Since $H_{0,x} \subset H_x^-$, we also have $H_x \subset H_{0,x}^+$ hence for all $A \in H_x$, $\scal{\NN_0(x)}{A} \geq a(x)$.

In particular, for $A = \nabla^2u(x)$, we obtain that
\begin{equation}\label{EQ:Borne inférieure square phi}
    \square u \geq a(x) \geq -m,
\end{equation}
where $m = -\min_{x \in X} a(x)$. Notice that inequality \eqref{EQ:Borne inférieure square phi} proves that $m$ must be non-negative.

From now on, we only use \eqref{EQ:Borne inférieure square phi} to obtain a bound on $\norme{u}_{L^1}$. In particular, we don't require any more that $u$ solves \eqref{EQ:Equation}. Let $G : X^2 \rightarrow \R \cup \{+\infty\}$ be the Green function associated to $-\square$. Recall that for all $y$,
\begin{equation}\label{EQ:Moyenne G = 0}
    \int_X G(x,y) \, d\mu_g(x) = 0,
\end{equation}
due to the fact that $\ker(\square) = \R 1$.

Unlike the Laplacian operator, $\square$ needs not be self-adjoint so we introduce a generator $k$ of $\ker(\square^*)$. In particular, this function satisfies
\begin{equation}\label{EQ:square*k = 0}
    \int_X k(x)\square u(x) \, d\mu_g(x) = \int_X \square^*k(x)u(x) \, d\mu_g(x) = 0.
\end{equation}
$\square$ depends only on $X$ and $\bH_0$, thus $G$ and $k$ depend only on $X$, $g$ and $\bH_0$.

Then, assume without loss of generality that $k$ takes at least one positive value, say at a point $x_0$. We show that $k$ is positive. First of all, if $k$ takes a negative value, say at a point $x_1 \in X$, we can easily build a positive function $f$ such that
\begin{equation}\label{EQ:f orthogonale à k}
    \int_X f(x)k(x) \, d\mu_g(x) = 0.
\end{equation}
For this, start from $f = 1$ and add a bump near $x_0$ if the integral of $k$ is negative, or a bump near $x_1$ else. With a well-adjusted bump, we obtain \eqref{EQ:f orthogonale à k}. This integral means that $f \in \ker(\square^*)^\bot = \im(\square)$ so there is a function $g$ such that $\square g = f > 0$. It contradicts the existence of a maximum for $g$.

It shows that $k \geq 0$ (and we still have $k(x_0) > 0$). If $k$ vanishes on $X$, there is a positive radius $r$ such that $k$ is positive on the geodesic ball $B$ of centre $x_0$ and radius $r$ but vanishes at a point of $\partial B$. Since $k|_B$ is a solution to a uniform elliptic equation, by the Harnack inequality, there is a constant $C_1$ that only depends on $B$, $X$ and $\square^*$ such that
$$
\sup_B k \leq C_1\inf_B k.
$$
However, $\inf_B k = 0$, which contradicts the fact that $k(x_0) > 0$. Thus, $k > 0$ on $X$.

Let $x_1 \in X$ be a point where $u$ reaches its maximum $0$. We have
\begin{align*}
    0 & = u(x_1)\\
    & = \frac{1}{\Vol_g(X)}\int_X u(x) d\mu_g(x) - \int_X G(x,x_1)\square u(x) \, d\mu_g(x)\\
    & \leq -\frac{1}{\Vol_g(X)}\norme{u}_{L^1} - \int_X G(x,x_1)\left(m + \square u(x)\right) \, d\mu_g(x) \textrm{ by \eqref{EQ:Moyenne G = 0}},\\
    & \leq -\frac{1}{\Vol_g(X)}\norme{u}_{L^1} + \int_X \frac{\max_{X^2} (-G)}{\min_X k}k(x)\left(m + \square u(x)\right) \, d\mu_g(x) \textrm{ by \eqref{EQ:Borne inférieure square phi}},\\
    & = -\frac{1}{\Vol_g(X)}\norme{u}_{L^1} + m\frac{\max_{X^2} (-G)}{\min k}\norme{k}_{L^1} \textrm{ by \eqref{EQ:square*k = 0}}.
\end{align*}
We deduce that
\begin{equation}\label{EQ:Borne norme L1}
    \norme{u}_{L^1} \leq C_2,
\end{equation}
with $C_2$ that depends only on $X$, $g$ and $\bH_0$.
\end{proof}

We can now bound the $\cC^0$ norm of $u$ by using the fact that $0$ is a $\cC$-sub-solution and the constants $\delta$ and $R$ introduced in \eqref{EQ:Définition delta et R}.

\begin{proposition}[$\cC^0$ estimates]\label{PRO:Estimées C0}
    There is a constant $C = C(X,g,\delta,R)$ such that for all smooth solutions $u$ of \eqref{EQ:Equation} with $\max_X u = 0$,
    $$
    \norme{u}_{\cC^0} \leq C\left(\norme{u}_{L^1} + 1\right).
    $$
\end{proposition}
\begin{proof}
The proof relies on the Alexandroff--Bakelman--Pucci maximum principle as used by Blocki \cite{Blocki} in the case of the complex Monge--Ampère equation and by Székeleyhidi \cite{Szekelyhidi} for more general complex Hessian equations.

Then, let $x_0 \in X$ be a point where $u$ reaches its minimum $\min_X u = -\norme{u}_{\cC^0}$. When $r > 0$, we denote by $B_r$ the Euclidean ball of centre $0$ and radius $r$ in $T_{x_0}X$. Let $r > 0$ small enough so
$$
\exp_{g,x_0} : B_r \longrightarrow X
$$
is a diffeomorphism on its image. Up to reducing $r$, we may assume that there are positive constants $C_1$, $C_2$ and $C_3$ such that for all $x \in B_r$ and $x \in \exp_{g,x_0}(B_r)$,
\begin{equation}\label{EQ:Borne dexp}
    \norme{d\exp_{g,x_0}(x)} \leq C_1,
\end{equation}
\begin{equation}\label{EQ:Borne dexp-1}
    \norme{d\exp_{g,x_0}^{-1}(x)} \leq C_2,
\end{equation}
\begin{equation}\label{EQ:Borne d^2exp-1}
    \norme{d^2\exp_{g,x_0}^{-1}(x)} \leq C_3.
\end{equation}
We used the operator norms. By compactness of $X$, we can choose the triple $(r,C_1,C_2,C_3)$ uniformly with respect to $x_0$. In other words, they only depend on $X$ and $g$.

Let the real function of $B_r$ defined by
$$
\psi : x \longmapsto u(\exp_{g,x_0}^{-1}(x)) + \epsilon\norme{x}^2.
$$
where
\begin{equation}\label{EQ:Définition epsilon}
    \epsilon = \frac{\delta}{2C_2^2 + \frac{C_3}{2} + 2rC_3}.
\end{equation}

We introduce
$$
P = \left\{a \in B_r \mid \norme{d\psi(a)} < \frac{\epsilon}{2} \textrm{ and } \forall b \in B_r, \psi(b) \geq \psi(a) + d\psi(a)(b - a)\right\},
$$
By \cite[Proposition 11]{Szekelyhidi}, since $u$ reaches its minimum at $x_0$, there is a constant $c_4$ that only depends on $n$, $r$, and $\epsilon$ such that
\begin{equation}\label{EQ:Borne inférieure intégrale sur P}
    \int_P \det(d^2\psi(a)) \, da \geq c_4
\end{equation}

Let $(x,\hat{x}) \in T\exp_{g,x_0}(B_r)$, $a = \exp_{g,x_0}^{-1}(x) \in B_r$ and $\hat{a} = d\exp_{g,x_0}^{-1}(x)\hat{x}$. We can compute that
\begin{equation}\label{EQ:Hessienne phi}
    d^2u(x)(\hat{x},\hat{x}) = d^2\psi(a)(\hat{a},\hat{a}) + d\psi(a)d^2\exp_{g,x_0}^{-1}(x)(\hat{x},\hat{x}) - 2\epsilon\norme{\hat{a}}^2 - 2\epsilon\scal{a}{d^2\exp_{g,x_0}^{-1}(x)\hat{x}}.
\end{equation}
If $a \in P$, then
\begin{equation}\label{EQ:Positivité d^2psi et borne dpsi}
    d^2\psi(a)(\hat{a},\hat{a}) \geq 0, \qquad \norme{d\psi(a)} < \frac{\epsilon}{2},
\end{equation}
We deduce from \eqref{EQ:Hessienne phi} and \eqref{EQ:Positivité d^2psi et borne dpsi} that
\begin{align*}
    d^2u(x)(\hat{x},\hat{x}) & > -\frac{\epsilon}{2}\norme{d^2\exp_{g,x_0}^{-1}(x)(\hat{x},\hat{x})} - 2\epsilon\norme{d\exp_{g,x_0}^{-1}(x)\hat{x}}^2 - 2\epsilon\norme{a}\norme{d^2\exp_{g,p_0}^{-1}(x)(\hat{x},\hat{x})}\\
    & \geq -\frac{\epsilon}{2}C_3\norme{\hat{x}}^2 - 2\epsilon C_2^2\norme{\hat{x}}^2 - 2\epsilon rC_3\norme{\hat{x}}^2 \textrm{ by \eqref{EQ:Borne dexp-1} and \eqref{EQ:Borne d^2exp-1}},\\
    & = -\delta\norme{\hat{x}}^2 \textrm{ by \eqref{EQ:Définition epsilon}}.
\end{align*}

We deduce that $\nabla^2u(x) > -\delta\Id$ when $x \in \exp_{g,x_0}(P)$. Since $u$ is a solution of \eqref{EQ:Equation}, by \eqref{EQ:Définition delta et R},
\begin{equation}\label{EQ:Borne Hessienne dans P}
    \norme{\nabla^2u(x)} \leq R,
\end{equation}
still when $x \in \exp_{g,x_0}(P)$.

Therefore,
\begin{align}
    d^2\psi(a)(\hat{a},\hat{a}) & = d^2u(x)(\hat{x},\hat{x}) - d\psi(a)d^2\exp_{g,x_0}^{-1}(x)(\hat{x},\hat{x}) + 2\epsilon\norme{\hat{a}}^2 + 2\epsilon\scal{a}{d^2\exp_{g,x_0}^{-1}(x)\hat{x}} \textrm{ by \eqref{EQ:Hessienne phi}}, \nonumber\\
    & \leq \norme{\nabla^2u(x)}\norme{\hat{x}}^2 + \frac{\epsilon}{2}C_3\norme{\hat{x}}^2 + 2\epsilon\norme{\hat{a}}^2 + 2\epsilon rC_3\norme{\hat{x}}^2 \textrm{ by \eqref{EQ:Borne dexp-1}, \eqref{EQ:Borne d^2exp-1} and \eqref{EQ:Positivité d^2psi et borne dpsi}}, \nonumber\\
    & \leq C_1^2R\hat{a}^2 + \frac{\epsilon}{2}C_1^2C_3\norme{\hat{a}}^2 + 2\epsilon\norme{\hat{a}}^2 + 2\epsilon rC_1^2C_3\norme{\hat{a}}^2 \textrm{ by \eqref{EQ:Borne dexp} and \eqref{EQ:Borne Hessienne dans P}}, \nonumber\\
    & \leq C_5\hat{a}^2, \label{EQ:Borne d^2psi dans P}
\end{align}
with $C_5 = C_1^2R + \frac{\epsilon}{2}C_1^2C_3 + 2\epsilon + 2\epsilon rC_1^2C_3$ that only depends on $X$, $g$, $\delta$ and $R$.

Finally, notice that for all $a \in P$,
\begin{equation}\label{EQ:psi(a) bornée par psi(0)}
    -\norme{u}_{\cC^0} = \psi(0) \geq \psi(a) - \frac{\epsilon r}{2},
\end{equation}
hence
\begin{align*}
    \norme{u}_{\cC^0} & \leq \frac{\norme{u}_{\cC^0}}{c_4}\int_P \det(d^2\psi(a)) \, da \textrm{ by \eqref{EQ:Borne inférieure intégrale sur P}},\\
    & \leq \frac{\norme{u}_{\cC^0}}{c_4}\int_P C_5^n \, da \textrm{ by \eqref{EQ:Borne d^2psi dans P}},\\
    & \leq \frac{C_5^n\norme{u}_{\cC^0}}{c_4}\int_P \frac{-\psi(a) + \frac{\epsilon r}{2}}{\norme{u}_{\cC^0}} \, dx \textrm{ by \eqref{EQ:psi(a) bornée par psi(0)}},\\
    & \leq \frac{C_5^nC_2^n}{c_4}\int_{\exp_{g,x_0}(P)} -u(x) - \epsilon\norme{\exp^{-1}(x)}^2 + \frac{\epsilon r}{2} \, d\mu_g(x) \textrm{ by \eqref{EQ:Borne dexp-1}},\\
    & \leq \frac{C_5^nC_2^n}{c_4}\left(\norme{u}_{L^1} + \frac{\epsilon r}{2}\Vol_g(X)\right).
\end{align*}
\end{proof}

\subsection{\texorpdfstring{$\cC^1$}{C1} estimates}\label{SEC:Estimées C1}

In this sub-section, we prove an \textit{a priori} estimate for $\norme{\nabla u}_{\cC^0}$ in function of $\norme{u}_{\cC^0}$. From now on, we assume that $\bH$ is almost spectral (or at least the first three points of Lemma \ref{LEM:Conséquences hypothèse}). We do not need to assume the existence of any $\cC$-sub-solution in this sub-section.

From now on, we write $F = F_\Id$. When $u$ solves \eqref{EQ:Equation}, we have for all $x$,
$$
F(x,\nabla^2u(x)) = 0.
$$
When we differentiate this equality, we obtain for all $\hat{x} \in T_xX$,
\begin{equation}\label{EQ:Dérivée F = 0}
    d_xF(x,\nabla^2u(x))\hat{x} + d_AF(x,\nabla^2u(x))d(\nabla^2u)(x)\hat{x} = 0.
\end{equation}

In local coordinates, \eqref{EQ:Dérivée F = 0} gives for all $x$ and $k$,
\begin{equation}\label{EQ:Dérivée F = 0 localement}
    F_{k,}(x,\nabla^2u(x)) + \sum_{i,j = 1}^n F_{,ij}(x,\nabla^2u(x))u_{ijk}(x) = 0.
\end{equation}

We also recall the bound we obtain from the rules of commutativity for the covariant derivatives at each point.
\begin{align}\label{EQ:Commutativité dérivées}
\begin{split}
    u_{ij} - u_{ij} & = 0,\\
    u_{ijk} - u_{jik} & = 0,\\
    \abs{u_{ijk} - u_{ikj}} & \leq B_1\norme{\nabla u},\\
    \abs{u_{kkij} - u_{ijkk}} & \leq B_2\norme{\nabla^2 u} + B_3\norme{\nabla u},
\end{split}
\end{align}
where $B_1$ only depends on the $\cC^0$ norm of the Riemannian tensor and $B_2$ on its $\cC^1$ norm.

Throughout the proofs of the different estimates, we will use the following consequences of Proposition \ref{PRO:Différentielles F_I},
\begin{equation}\label{EQ:nabla F borné}
    \sum_{i = 1}^n F_{,ii} = 1, \qquad 0 < \sum_{i,j = 1}^n F_{,ij} \leq n.
\end{equation}

Let $\underline{M} = (M_{10},M_{11},M_{12},M_{13},M_{14},M_{15})$ given by Lemma \ref{LEM:Conséquences hypothèse}.

\begin{proposition}[$\cC^1$ estimates]\label{PRO:Estimées C1}
    There is a constant $C = C(X,g,M_{10},M_{11},M_{12},M_{13})$ such that
    $$
    \norme{\nabla u}_{\cC^0} \leq C\left(\norme{u}_{\cC^0}^2 + 1\right).
    $$
\end{proposition}
\begin{proof}
The proof is based on the introduction of an auxiliary function and the fact that this function must have a maximum by compactness of $X$. We prove it in five steps.\\

\noindent\textbf{Step 1:} Introducing the function $f$ and its maximum.

Let
$$
f : x \longmapsto \norme{\nabla u(x)} + \frac{1}{2}(u(x) + A)^2,
$$
where
\begin{equation}\label{EQ:Estimée C1 définition A}
    A = \norme{u}_{\cC^0} + \sqrt{M_{11}(nB_1 + 1)}
\end{equation}

Usually, when such an auxiliary function is introduced to show a $\cC^1$-estimate, the function in $u$ (here, $\frac{1}{2}(u(x) + A)^2$) is a decreasing function. See for example \cite{Blocki}\cite{Zhang}. Here, we make it increasing so the gradient of $u$ will be an eigenvector of its Hessian associated with a negative eigenvalue with large absolute value. We then use the first point of Lemma \ref{LEM:Conséquences hypothèse} which, roughly speaking, says that we have uniform ellipticity of the equation in the direction of the smallest eigenvalue of the Hessian.

Notice that $A \geq \norme{u}_{\cC^0} + 1$ so for all $x$, $1 \leq u(x) + A \leq A$.

Let $x_0$ be a point where $f$ reaches its maximum. We have
\begin{align}
    \norme{\nabla u}_{\cC^0} & = \max_{x \in X} f(x) - \frac{1}{2}(u(x) + A)^2 \nonumber\\
    & \leq f(x_0) \nonumber\\
    & \leq \norme{\nabla u(x_0)} + \frac{1}{2}\left(\norme{u}_{\cC^0} + \sqrt{M_{11}(nB_1 + 1)}\right)^2. \label{EQ:Norme C1 en fonction de la valeur en x_0}
\end{align}
We deduce that it is enough to bound $\frac{\norme{\nabla u(x_0)}}{\norme{u}_{\cC^0}^2 + 1}$.\\

\noindent\textbf{Step 2:} Computing the first differential of $f$ and showing that the gradient is an eigenvector of the Hessian matrix.

If $\nabla u(x_0) = 0$, we are finished. Else, $f$ is smooth in an open neighbourhood of $x_0$ and locally around this point
$$
f(x) = \left(\sum_{k = 1}^n u_k(x)^2\right)^{1/2} + \frac{1}{2}(u(x) + A)^2,
$$
thus for all $i$,
\begin{align}
    f_i(x) & = \left(\sum_{k = 1}^n u_k(x)^2\right)^{-1/2}\sum_{k = 1}^n u_{ki}(x)u_k(x) + (u(x) + A)u_i(x) \label{EQ:f_i estimées C1}\\
    & = \frac{1}{\norme{\nabla u(x_0)}}\sum_{k = 1}^n u_{ki}(x_0)u_k(x_0) + (u(x) + A)u_i(x_0) \textrm{ at } x = x_0, \nonumber\\
    & = 0 \textrm{ because $f$ is maximal at } x_0. \label{EQ:f_i = 0 estimées C1}
\end{align}

Let $\lambda_1 \geq \cdots \geq \lambda_n$ be the eigenvalues of $\nabla^2u(x_0)$. Equality \eqref{EQ:f_i = 0 estimées C1} exactly tells that $\nabla u(x_0)$ is an eigenvector of $\nabla^2u(x_0)$. We may choose the local coordinates so there is an index $r$ such that
$$
u_i(x_0) = \delta_{ir}\norme{\nabla u(x_0)}, \qquad u_{ij}(x_0) = \delta_{ij}\lambda_i.
$$
Moreover, \eqref{EQ:f_i = 0 estimées C1} at $i = r$ yields
\begin{equation}\label{EQ:lambda_r}
    \lambda_r = -(u(x) + A)\norme{\nabla u(x_0)} < 0.
\end{equation}

If $\norme{\nabla u(x_0)} \leq M_{10}$, we are finished. Else, $\lambda_r \leq -M_{10}$ so $\norme{\nabla^2u(x_0)} \geq M_{10}$ and Lemma \ref{LEM:Conséquences hypothèse} applies at $(x_0,\nabla^2u(x_0))$.

\noindent\textbf{Step 3:} Computing the second differential of $f$.

From now on, when we evaluate an expression at $x = x_0$, we omit to write it to lighten the expressions. From \eqref{EQ:f_i estimées C1}, we compute that

\begin{align*}
    f_{ij} & = \left(\sum_{k = 1}^n u_k(x)^2\right)^{-1/2}\sum_{k = 1}^n u_{kij}(x)u_k(x) + \left(\sum_{k = 1}^n u_k(x)^2\right)^{-1/2}\sum_{k = 1}^n u_{ki}(x)u_{kj}(x)\\
    & -\left(\sum_{k = 1}^n u_k(x)^2\right)^{-3/2}\left(\sum_{k = 1}^n u_{ki}(x)u_k(x)\right)\left(\sum_{k = 1}^n u_{kj}(x)u_k(x)\right)\\
    & + (u(x) + A)u_{ij}(x) + u_i(x)u_j(x)\\
    & = u_{rij} + \frac{\delta_{ij}}{u_r}\lambda_i^2 - \frac{\delta_{ir}\delta_{jr}}{u_r}\lambda_r^2 + \delta_{ij}(u + A)\lambda_i + \delta_{ir}\delta_{jr}u_r^2 \textrm{ at } x = x_0,\\
    & = u_{rij} + \frac{\delta_{ij}}{u_r}\lambda_i^2 - \delta_{ir}\delta_{jr}(u + A)^2u_r\\
    & + \delta_{ij}(u + A)\lambda_i + \delta_{ir}\delta_{jr}u_r^2 \textrm{ by \eqref{EQ:lambda_r}}.
\end{align*}

Since $f$ reaches a maximum at $x_0$, its Hessian matrix at this point is negative. By ellipticity, the matrix $(F_{,ij})$ (at $(x_0,\nabla^2u(x_0))$) is positive definite. In particular, we deduce that
\begin{align}
    0 & \geq \sum_{i,j = 1}^n F_{,ij}f_{ij} \nonumber\\
    & = \sum_{i,j = 1}^n F_{,ij}u_{rij} + (u + A)\sum_{i = 1}^n F_{ii}\lambda_i \label{EQ:Estimée C1 terme 1}\\
    & + \frac{1}{u_r}\sum_{i = 1}^n F_{ii}\lambda_i^2 \label{EQ:Estimée C1 terme 2}\\
    & - F_{rr}(u + A)^2u_r + F_{rr}u_r^2. \label{EQ:Estimée C1 terme 3}
\end{align}

Finally, we bound from below each line of the above equation in order to bound $u_r = \norme{\nabla u}$.\\

\noindent\textbf{Step 4:} Bounding \eqref{EQ:Estimée C1 terme 1}, \eqref{EQ:Estimée C1 terme 2} and \eqref{EQ:Estimée C1 terme 3}.

For the first term, we have
\begin{align}
    & \quad\ \sum_{i,j = 1}^n F_{,ij}u_{rij} + (u + A)\sum_{i = 1}^n F_{ii}\lambda_i\\
    & \geq \sum_{i,j = 1}^n F_{,ij}u_{rij} - M_{12}(u + A) \textrm{ by point 2 of Lemma \ref{LEM:Conséquences hypothèse}}, \nonumber\\
    & \geq \sum_{i,j = 1}^n F_{,ij}u_{ijr} - nB_1u_r - M_{12}(u + A) \textrm{ by \eqref{EQ:Commutativité dérivées} and \eqref{EQ:nabla F borné}}, \nonumber\\
    & = -F_{r,} - nB_1u_r - M_{12}(u + A) \textrm{ by \eqref{EQ:Dérivée F = 0 localement}}, \nonumber\\
    & \geq -M_{13} - nB_1u_r - M_{12}(u + A) \textrm{ by point 3 of Lemma \ref{LEM:Conséquences hypothèse}}, \nonumber\\
    & \geq -nB_1u_r - M_{12}A - M_{13}. \label{EQ:Borne estimée C1 terme 1}
\end{align}

The second term is very positive, especially when $u_r$ becomes large so it will give us the final bound on $u_r$.
\begin{align}
    \frac{1}{u_r}\sum_{i = 1}^n F_{ii}\lambda_i^2 & \geq \frac{1}{u_r}F_{nn}\lambda_n^2 \nonumber\\
    & \geq \frac{1}{M_{11}u_r}\lambda_r^2 \textrm{ by point 1 of Lemma \ref{LEM:Conséquences hypothèse} and the fact that } \lambda_n \leq \lambda_r \leq 0, \nonumber\\
    & = \frac{1}{M_{11}}(u + A)^2u_r \textrm{ by \eqref{EQ:lambda_r}}, \nonumber\\
    & \geq \frac{1}{M_{11}}(-\norme{u}_{\cC^0} + A)^2u_r. \label{EQ:Borne estimée C1 terme 2}
\end{align}

The third term can easily be bounded as
\begin{equation}\label{EQ:Borne estimée C1 terme 3}
    -F_{rr}(u + A)^2u_r + F_{rr}u_r^2 \geq (u_r - A^2)F_{rr}u_r
\end{equation}

\noindent\textbf{Step 5:} Conclusion.

Finally, by \eqref{EQ:Borne estimée C1 terme 1}, \eqref{EQ:Borne estimée C1 terme 2} and \eqref{EQ:Borne estimée C1 terme 3},
\begin{align*}
    0 & \geq \sum_{i,j = 1}^n F_{,ij}f_{ij}\\
    & \geq -nB_1u_r - M_{12}A - M_{13} + \frac{1}{M_{11}}(-\norme{u}_{\cC^0} + A)^2u_r + (u_r - A^2)F_{rr}u_r\\
    & = u_r - M_{12}\left(\norme{u}_{\cC^0} + \sqrt{M_{11}(nB_1 + 1)}\right) - M_{13}\\
    & + \left(u_r - \left(\norme{u}_{\cC^0} + \sqrt{M_{11}(nB_1 + 1)}\right)^2\right)F_{rr}u_r \textrm{ by \eqref{EQ:Estimée C1 définition A}}.
\end{align*}
We easily deduce that if $\frac{u_r}{\norme{u}_{\cC^0}^2 + 1}$ is large enough, both lines in the above equation will be positive, leading to a contradiction. We conclude by using the fact that $\norme{\nabla u} = u_r$ and \eqref{EQ:Norme C1 en fonction de la valeur en x_0}.
\end{proof}

\subsection{\texorpdfstring{$\cC^2$}{C2} estimates}\label{SEC:Estimées C2}

In this sub-section, we prove an \textit{a priori} estimate for $\norme{\nabla^2u}_{\cC^0}$ in function of $\norme{\nabla u}_{\cC^0}$ and $\norme{u}_{\cC^0}$. We assume that $\bH$ is almost spectral (or at least the results of Lemma \ref{LEM:Conséquences hypothèse}) to hold. We also assume that $0$ is a $\cC$-sub-solution so \eqref{EQ:Définition delta et R} holds for some positive constants $\delta$ and $R$.

When we differentiate twice the equality $F(x,\nabla^2u(x)) = 0$ in the direction $\hat{x} \in T_xX$, we obtain that
$$
d^2F(x,\nabla^2u(x))((\hat{x},d(\nabla^2u)(x)\hat{x}),(\hat{x},d(\nabla^2u)(x)\hat{x})) + d_AF(x,\nabla^2u(x))(d^2(\nabla^2u)(x)(\hat{x},\hat{x})) = 0.
$$
Thus, by point 4 of Lemma \ref{LEM:Conséquences hypothèse},
\begin{adjustwidth}{-5cm}{5cm}
\begin{align}
    \sum_{i,j = 1}^n F_{,ij}(x_0,\nabla^2u(x_0))u_{ij11}(x_0) & = d_AF(x,\nabla^2u(x))(d^2(\nabla^2u)(x)(\hat{x},\hat{x})) \nonumber\\
    & = -d^2F(x,\nabla^2u(x))((\hat{x},d(\nabla^2u)(x)\hat{x}),(\hat{x},d(\nabla^2u)(x)\hat{x})) \nonumber\\
    & \geq -M_{14}\norme{\nabla^2u(x_0)} + \frac{1}{\lambda_1 - \lambda_n}\sum_{i,j = 1}^n (F_{,ij}(x_0,\nabla^2u(x_0)) - M_{15}F_{,11})u_{i11}(x_0)u_{j11}(x_0). \label{EQ:Borne inférieure d_AF localement}
\end{align}
\end{adjustwidth}

We are now ready to prove Theorem \ref{THEI:Estimées C2}. Let $\underline{M}'$ given by Lemma \ref{LEM:Conséquences hypothèse}.

\begin{theorem}[$\cC^2$ estimates]\label{THE:Estimées C2}
    There is a constant $C = C(X,g,\bH_0,\delta,R,\underline{M}')$ such that
    $$
    \norme{\nabla^2u}_{\cC^0} \leq \exp(C(\norme{u}_{\cC^0} + 1))\left(\norme{\nabla u}_{\cC^0}^2 + 1\right).
    $$
\end{theorem}
\begin{proof}
The proof is once again based on studying what happens when a suitable auxiliary function reaches its maximum. The introduced auxiliary function and the method used are the same as that used in \cite{Chou_Wang}\cite{Hou_Ma_Wu}\cite{Tosatti_Weinkove} for example. It is also a simplified version of Székelyhidi's proof \cite{Szekelyhidi}, where we don't have to deal with torsion.

During all the proof, we denote by $C_1,C_2,C_3,\ldots$ constants that depend only on $X$, $g$, $\bH_0$, $\delta$, $R$, $M_{10}$, $M_{11}$, $M_{12}$, $M_{13}$, $M_{14}$ and $M_{15}$. We prove the inequality in ten steps.\\

\noindent\textbf{Step 1:} Bounding the lowest eigenvalue of the Hessian matrix.

For all $x \in X$, let $\lambda_1(x) \geq \cdots \geq \lambda_n(x)$ be the eigenvalues of $\nabla^2u(x)$. Using the same notations as in the proof of Proposition \ref{PRO:Estimées L1}, we have for all $x$,
\begin{equation}\label{EQ:Borne inférieure laplacien H_0}
    \scal{\nabla^2u(x)}{\NN_0(x)} \geq -m,
\end{equation}
where $\NN_0$ and $m$ only depend on $X$ and $\bH_0$. Let $c_1$ and $C_2$ be constants such that for all $x$, $c_1\Id \leq \NN_0(x) \leq C_2\Id$. If $\lambda_n(x) \leq 0$, \eqref{EQ:Borne inférieure laplacien H_0} implies that
$$
-m \leq \scal{\nabla^2u(x)}{\NN_0(x)} \leq (n - 1)C_2\lambda_1(x) + c_1\lambda_n(x),
$$
hence
\begin{equation}\label{EQ:Norme en fonction de lambda_1}
    \norme{\nabla^2u(x)} \leq \sqrt{n}\max\{\lambda_1(x),-\lambda_n(x)\} \leq \sqrt{n}\max\left\{\lambda_1(x),\frac{m + (n - 1)C_2\lambda_1(x)}{c_1}\right\}.
\end{equation}
It means that it is enough to control $\lambda_1$ from above in order to control the norm. Let $\lambda_{\max} = \max_{x \in X} \lambda_1(x)$.\\

\noindent\textbf{Step 2:} Introducing the function $f$ and its maximum.

Let
$$
f : x \longmapsto \ln(\lambda_1(x)) + \phi\!\left(\frac{1}{2}\norme{\nabla u(x)}^2\right) - \psi(u(x)),
$$
with the convention that $f(x) = -\infty$ if $\lambda_1(x) \leq 0$, where
$$
\phi(s) = -\frac{1}{A_1}\ln\left(1 - \frac{s}{K}\right), \qquad \psi(t) = -\frac{A_2}{2\norme{u}_{\cC^0}}t^2 + A_2t
$$
with $K = \norme{\nabla u}_{\cC^0}^2 + 1$ and $A_1,A_2 \geq 2$ are constants to be determined later. Notice that when $0 \leq s \leq \frac{1}{2}K$ and $-\norme{u}_{\cC^0} \leq t \leq 0$,
\begin{equation}\label{EQ:Bornes phi}
    0 \leq \phi \leq \ln(2), \qquad \frac{1}{A_1K} \leq \phi' \leq \frac{2}{A_1K} \leq \frac{1}{K}, \qquad \phi'' = A_1(\phi')^2,
\end{equation}
\begin{equation}\label{EQ:Bornes psi}
    -\frac{3}{2}A_2\norme{u}_{\cC^0} \leq \psi \leq 0, \qquad A_2 \leq \psi' \leq 2A_2, \qquad -\psi'' \geq \frac{1}{4A_2\norme{u}_{\cC^0}}(\psi')^2.
\end{equation}

$f$ is continuous as a function with values in $\R \cup \{-\infty\}$ so it admits a maximum. Let $x_0$ be a point where $f$ reaches its maximum. We have
\begin{align}
    \lambda_{\max} & \leq \max_{x \in X} \exp\left(f(x) - \phi\!\left(\frac{1}{2}\norme{\nabla u(x)}^2\right) + \psi(u(x))\right) \nonumber\\
    & \leq \exp(f(x_0)) \textrm{ by \eqref{EQ:Bornes phi} and \eqref{EQ:Bornes psi}}, \nonumber\\
    & \leq \lambda_1(x_0)^+\exp\left(-\phi\!\left(\frac{1}{2}\norme{\nabla u(x_0)}^2\right) + \psi(u(x_0))\right) \nonumber\\
    & \leq 2\exp\left(\frac{3}{2}A_2\norme{u}_{\cC^0}\right)\lambda_1(x_0)^+ \textrm{ by \eqref{EQ:Bornes phi} and \eqref{EQ:Bornes psi}}, \label{EQ:lambda_max en fonction de lambda_1}
\end{align}
so it is enough to bound $\frac{\lambda_1(x_0)^+}{K}$ to deduce the wanted inequality.\\

\noindent\textbf{Step 3:} Computing the first differential of $f$.

Assume without loss of generality that $\lambda_1(x_0) \geq K > 0$ and $\lambda_1(x_0) \geq M_{10}$ so the assumptions on $\bH$ apply at $x_0$. Assume also that $\lambda_n(x_0) < 0$. Indeed, if $\lambda_n(x_0) \geq 0$, then $\nabla^2u(x_0)$ is positive and can immediately be bounded by $R$ since $0$ is a $\cC$-sub-solution. Consider a local orthonormal frame around $x_0$ in which $\nabla^2u(x_0)$ is diagonal with $u_{ii}(x_0) = \lambda_i(x_0)$. For all $x$ near $x_0$, $u_{11}(x) \leq \lambda_1(x)$ with equality at $x_0$ so the function
$$
\tilde{f} : x \longmapsto \ln(u_{11}(x)) + \phi\!\left(\frac{1}{2}\norme{\nabla u(x)}^2\right) - \psi(u(x))
$$
is well-defined near $x_0$ and reaches a maximum at this point. Moreover, $\tilde{f}$ is smooth on its domain and for all $i$,
\begin{align}
    \tilde{f}_i(x) & = \frac{u_{11i}(x)}{u_{11}(x)} + \phi'\!\left(\frac{1}{2}\norme{\nabla u(x)}^2\right)\sum_{k = 1}^n u_{ki}(x)u_k(x) - \psi'(u(x))u_i(x) \label{EQ:f_i estimées C2}\\
    & = \frac{u_{11i}(x_0)}{\lambda_1(x_0)} + \phi'\!\left(\frac{1}{2}\norme{\nabla u(x)}^2\right)\lambda_i(x_0)u_i(x_0) - \psi'(u(x_0))u_i(x_0) \textrm{ at } x = x_0, \nonumber\\
    & = 0 \textrm{ because $\tilde{f}$ is maximal at } x_0. \label{EQ:f_i = 0 estimées C2}
\end{align}

\noindent\textbf{Step 4:} Computing the second differential of $\tilde{f}$.

From now on, when we evaluate an expression at $x = x_0$, we omit to write it to lighten the expressions. Similarly, we write $\psi'$ instead of $\psi'(u(x_0))$ for example. From \eqref{EQ:f_i estimées C2}, we compute that
\begin{align*}
    \tilde{f}_{ij} & = \frac{u_{11ij}(x)}{u_{11}(x)} - \frac{u_{11i}(x)u_{11j}(x)}{u_{11}(x)^2} + \phi'\!\left(\frac{1}{2}\norme{\nabla u(x)}^2\right)\sum_{k = 1}^n u_{kij}(x)u_k(x)\\
    & + \phi'\!\left(\frac{1}{2}\norme{\nabla u(x)}^2\right)\sum_{k = 1}^n u_{ki}(x)u_{kj}(x) + \phi''\!\left(\frac{1}{2}\norme{\nabla u(x)}^2\right)\left(\sum_{k = 1}^n u_{ki}(x)u_k(x)\right)\left(\sum_{k = 1}^n u_{kj}(x)u_k(x)\right)\\
    & - \psi'(u(x))u_{ij}(x) - \psi''(u(x))u_i(x)u_j(x)\\
    & = \frac{u_{11ij}}{\lambda_1} - \frac{u_{11i}u_{11j}}{\lambda_1^2} + \phi'\sum_{k = 1}^n u_{kij}u_k + \delta_{ij}\phi'\lambda_i^2 + \phi''\lambda_i\lambda_ju_iu_j - \delta_{ij}\psi'\lambda_i - \psi''u_iu_j \textrm{ at } x = x_0.
\end{align*}
Since $\tilde{f}$ reaches a maximum at $x_0$, its Hessian matrix at this point is negative. By ellipticity, the matrix $(F_{,ij})$ (at $(x_0,\nabla^2u(x_0))$) is positive definite. In particular, we deduce that
\begin{align}
    0 & \geq \sum_{i,j = 1}^n F_{,ij}\tilde{f}_{ij} \nonumber\\
    & = \sum_{i,j = 1}^n F_{,ij}\frac{u_{11ij}}{\lambda_1} - \sum_{i,j = 1}^n F_{,ij}\frac{u_{11i}u_{11j}}{\lambda_1^2} + \phi'\sum_{i,j,k = 1}^n F_{,ij}u_{kij}u_k \label{EQ:Terme à borner 1}\\
    & + \phi'\sum_{i = 1}^n F_{,ii}\lambda_i^2 + \phi''\sum_{i,j = 1}^n F_{,ij}\lambda_i\lambda_ju_iu_j - \psi'\sum_{i = 1}^n F_{,ii}\lambda_i - \psi''\sum_{i,j = 1}^n F_{,ij}u_iu_j.
\end{align}

\noindent\textbf{Step 5:} Bounding \eqref{EQ:Terme à borner 1}.

First,
\begin{align}
    \sum_{i,j = 1}^n F_{,ij}u_{11ij} & \geq -nB_2\norme{\nabla^2u} + \sum_{i,j = 1}^n F_{,ij}u_{ij11} \textrm{ by \eqref{EQ:Commutativité dérivées}}, \nonumber\\
    & \geq -(nB_1 + M_{14})\norme{\nabla^2u} + \frac{1}{\lambda_1 - \lambda_n}\sum_{i,j = 1}^n (F_{,ij} - M_{15}F_{,11})u_{i11}u_{j11}\textrm{ by \eqref{EQ:Borne inférieure d_AF localement}} \nonumber\\
    & \geq -C_3\lambda_1 + \frac{1}{\lambda_1 - \lambda_n}\sum_{i,j = 1}^n (F_{,ij} - M_{15}F_{,11})u_{i11}u_{j11} \textrm{ by \eqref{EQ:Norme C1 en fonction de la valeur en x_0}}. \label{EQ:Inégalité estimées C2 1}
\end{align}

Then,
\begin{align}
    \abs{u_{i11}u_{j11} - u_{11i}u_{11j}} & \leq \abs{u_{i11}u_{j11} - u_{11i}u_{j11}} + \abs{u_{11i}u_{j11} - u_{11i}u_{11j}} \nonumber\\
    & \leq B_1\sqrt{K}(\abs{u_{j11}} + \abs{u_{11j}}) \textrm{ by \eqref{EQ:Commutativité dérivées}}, \nonumber\\
    & \leq B_1\sqrt{K}(2\abs{u_{11j}} + B_1\sqrt{K}) \textrm{ by \eqref{EQ:Commutativité dérivées}}, \nonumber\\
    & = B_1^2K - 2B_1\sqrt{K}\lambda_1\abs{\phi'\lambda_ju_j - \psi'u_j} \textrm{ by \eqref{EQ:f_i = 0 estimées C2}}, \nonumber\\
    & \leq B_1^2K + 2B_1\lambda_1^2 + 4B_1A_2K \textrm{ by \eqref{EQ:Bornes phi} and \eqref{EQ:Bornes psi}}, \nonumber\\
    & \leq C_4\lambda_1\left(\lambda_1 + A_2\right) \textrm{ because } \lambda_1 \geq K. \label{EQ:Inégalité estimées C2 2}
\end{align}

Therefore,
\begin{align}
    \sum_{i,j = 1}^n F_{,ij}\frac{u_{11ij}}{\lambda_1} & \geq -C_3 + \frac{1}{\lambda_1(\lambda_1 - \lambda_n)}\sum_{i,j = 1}^n (F_{,ij} - M_{15}F_{,11})u_{i11}u_{j11} \textrm{ by \eqref{EQ:Inégalité estimées C2 1}} \nonumber\\
    & \geq -C_3 - C_5\frac{\lambda_1 + A_2}{\lambda_1 - \lambda_n} + \frac{1}{\lambda_1(\lambda_1 - \lambda_n)}\sum_{i,j = 1}^n (F_{,ij} - M_{15}F_{,11})u_{11i}u_{11j} \textrm{ by \eqref{EQ:Inégalité estimées C2 2} and \eqref{EQ:nabla F borné}}, \nonumber\\
    & \geq -C_3 - C_5\frac{\lambda_1 + A_2}{\lambda_1} - M_{15}F_{,11}\left(\sum_{i = 1}^n \frac{u_{11i}}{\lambda_1}\right)^2 + \left(1 - \frac{\abs{\lambda_n}}{\lambda_1}\right)\sum_{i,j = 1}^n F_{,ij}\frac{u_{11i}u_{11j}}{\lambda_1^2}. \label{EQ:Inégalité estimées C2 3}
\end{align}

We then bound the third term in \eqref{EQ:Terme à borner 1}.
\begin{align}
    \phi'\sum_{i,j,k = 1}^n F_{,ij}u_{kij}u_k & \geq -nB_1\phi'K + \phi'\sum_{i,j,k = 1}^n F_{,ij}u_{ijk}u_k \textrm{ by \eqref{EQ:Commutativité dérivées} and \eqref{EQ:nabla F borné}}, \nonumber\\
    & \geq -nB_1 + \phi'\sum_{k = 1}^n F_{k,}u_k \textrm{ by \eqref{EQ:Bornes phi} and \eqref{EQ:Dérivée F = 0 localement}}, \nonumber\\
    & \geq -nB_1 - \frac{nM_{13}}{\sqrt{K}} \textrm{ by point 3 of Lemma \ref{LEM:Conséquences hypothèse} and \eqref{EQ:Bornes phi}}, \nonumber\\
    & \geq -C_6 \textrm{ because } K \geq 1. \label{EQ:Inégalité estimées C2 4}
\end{align}

Finally, combining \eqref{EQ:Inégalité estimées C2 3} and \eqref{EQ:Inégalité estimées C2 4}, we obtain that
\begin{align}
    & \quad\ \sum_{i,j = 1}^n F_{,ij}\frac{u_{11ij}}{\lambda_1} - \sum_{i,j = 1}^n F_{,ij}\frac{u_{11i}u_{11j}}{\lambda_1^2} + \phi'\sum_{i,j,k = 1}^n F_{,ij}u_{kij}u_k \nonumber\\
    & \geq -C_3 - C_5\frac{\lambda_1 + A_2}{\lambda_1} - M_{15}F_{,11}\left(\sum_{i = 1}^n \frac{u_{11i}}{\lambda_1}\right)^2 + \left(1 - \frac{\abs{\lambda_n}}{\lambda_1}\right)\sum_{i,j = 1}^n F_{,ij}\frac{u_{11i}u_{11j}}{\lambda_1^2} - \sum_{i,j = 1}^n F_{,ij}\frac{u_{11i}u_{11j}}{\lambda_1^2} - C_6 \nonumber\\
    & \geq -C_7\frac{\lambda_1 + A_2}{\lambda_1} - M_{15}F_{,11}\left(\sum_{i = 1}^n \frac{u_{11i}}{\lambda_1}\right)^2 - \frac{\abs{\lambda_n}}{\lambda_1}\sum_{i,j = 1}^n F_{,ij}\frac{u_{11i}u_{11j}}{\lambda_1^2} \nonumber\\
    & = -C_7\frac{\lambda_1 + A_2}{\lambda_1} - M_{15}F_{,11}\left(\sum_{i = 1}^n \phi'\lambda_iu_i - \psi'u_i\right)^2 - \frac{\abs{\lambda_n}}{\lambda_1}\sum_{i,j = 1}^n F_{,ij}(\phi'\lambda_iu_i - \psi'u_i)(\phi'\lambda_ju_j - \psi'u_j) \textrm{ by \eqref{EQ:f_i = 0 estimées C2}}. \label{EQ:Inégalité estimées C2 5}
\end{align}

We then do a case disjunction as in \cite{Szekelyhidi}. Let $\kappa$ given by Proposition \ref{PRO:Guan} and $\varepsilon$ a positive constant that we determine later. It will play the role of $\delta$ in \cite{Szekelyhidi}. We will choose first $A_1$ in function of $\varepsilon$, then $A_2$ independently from $A_1$ and $\varepsilon$ and finally, we will choose $\varepsilon$ in function of $A_2$.\\

\noindent\textbf{Step 6:} In this step, we assume either that $F_{,11} \geq \frac{\varepsilon^2}{M_{11}}$ or $\abs{\lambda_n} \geq \varepsilon\lambda_1$. We show a bound on $\lambda_1$.

If $F_{,11} \geq \frac{\varepsilon^2}{M_{11}}$,
\begin{equation}\label{EQ:Ellipticité uniforme}
    \sum_{i = 1}^n F_{,ii}\lambda_i^2 \geq F_{,11}\lambda_1^2 \geq \frac{\varepsilon^2}{M_{11}}\lambda_1^2.
\end{equation}
If $\abs{\lambda_n} \geq \varepsilon\lambda_1$, then by point 1 of Lemma \ref{LEM:Conséquences hypothèse},
$$
\sum_{i = 1}^n F_{,ii}\lambda_i^2 \geq F_{,nn}\lambda_n^2 \geq \frac{\varepsilon^2}{M_{11}}\lambda_1^2.
$$
In any case, \eqref{EQ:Ellipticité uniforme} holds.

Using \eqref{EQ:nabla F borné} and the fact that $\frac{\abs{\lambda_n}}{\lambda_1}$ is bounded (by \eqref{EQ:Norme en fonction de lambda_1}), \eqref{EQ:Inégalité estimées C2 5} becomes
\begin{align}
    & \quad\ \sum_{i,j = 1}^n F_{,ij}\frac{u_{11ij}}{\lambda_1} - \sum_{i,j = 1}^n F_{,ij}\frac{u_{11i}u_{11j}}{\lambda_1^2} + \phi'\sum_{i,j,k = 1}^n F_{,ij}u_{kij}u_k \nonumber\\
    & \geq -C_7\frac{\lambda_1 + A_2}{\lambda_1} - C_9\sum_{i = 1}^n (\phi'\lambda_iu_i - \psi'u_i)^2 \nonumber\\
    & \geq -C_7\frac{\lambda_1 + A_2}{\lambda_1} - \frac{C_{10}}{A_1^2K}\lambda_1^2 - C_{11}A_2^2K \textrm{ by \eqref{EQ:Bornes phi} and \eqref{EQ:Bornes psi}}. \label{EQ:Inégalité estimées C2 6}
\end{align}

Therefore,
\begin{align}
    0 & \geq \sum_{i,j = 1}^n F_{,ij}\tilde{f}_{ij} \nonumber\\
    & \geq -C_7\frac{\lambda_1 + A_2}{\lambda_1} - \frac{C_{10}}{A_1^2K}\lambda_1^2 - C_{11}A_2^2K \nonumber\\
    & + \phi'\sum_{i = 1}^n F_{,ii}\lambda_i^2 + \phi''\sum_{i,j = 1}^n F_{,ij}\lambda_i\lambda_ju_iu_j - \psi'\sum_{i = 1}^n F_{,ii}\lambda_i - \psi''\sum_{i,j = 1}^n F_{,ij}u_iu_j \textrm{ by \eqref{EQ:Inégalité estimées C2 6}} \nonumber\\
    & \geq -C_7\frac{\lambda_1 + A_2}{\lambda_1} - \frac{C_{10}}{A_1^2K}\lambda_1^2 - C_{11}A_2^2K + \frac{1}{A_1K}\sum_{i = 1}^n F_{,ii}\lambda_i^2 - 2A_2R \textrm{ by \eqref{EQ:Bornes phi}, \eqref{EQ:Bornes psi} and Proposition \ref{PRO:Guan}}, \nonumber\\
    & \geq -C_7\frac{\lambda_1 + A_2}{\lambda_1} - \frac{C_{10}}{A_1^2K}\lambda_1^2 - C_{11}A_2^2K + \frac{\varepsilon^2}{A_1M_{11}K}\lambda_1^2 - 2A_2R \textrm{ by \eqref{EQ:Ellipticité uniforme}}, \nonumber\\
    & \geq - C_{12}A_2^2K + \left(\frac{\varepsilon^2}{A_1M_{11}K} - \frac{C_{10}}{A_1^2K}\right)\lambda_1^2. \label{EQ:Inégalité estimées C2 7}
\end{align}
We choose
$$
A_1 = \frac{2C_{10}M_{11}}{\varepsilon^2},
$$
and we deduce from \eqref{EQ:Inégalité estimées C2 7} that
\begin{equation}\label{EQ:Borne finale lambda_1 cas 1}
    \lambda_1 \leq \frac{\sqrt{2C_{10}C_{12}}A_2M_{11}}{\varepsilon^2}K.
\end{equation}

\noindent\textbf{Step 7:} From now on, we assume that $F_{,11} < \frac{\varepsilon^2}{M_{11}}$ and $\abs{\lambda_n} \leq \varepsilon\lambda_1$ because else, \eqref{EQ:Borne finale lambda_1 cas 1} holds. We similarly show a bound on $\lambda_1$.

We choose $\varepsilon$ small enough so $\frac{\varepsilon^2}{M_{11}} \leq \kappa$. In particular, $F_{,11} < \kappa$ so by Proposition \ref{PRO:Guan}, we have
\begin{equation}\label{EQ:Guan}
    -\sum_{i = 1}^n F_{,ii}\lambda_i \geq \kappa.
\end{equation}

Therefore,
\begin{align}
    0 & \geq \sum_{i,j = 1}^n F_{,ij}\tilde{f}_{ij} \nonumber\\
    & \geq -C_7\frac{\lambda_1 + A_2}{\lambda_1} - M_{15}F_{,11}\left(\sum_{i = 1}^n \phi'\lambda_iu_i - \psi'u_i\right)^2 - \frac{\abs{\lambda_n}}{\lambda_1}\sum_{i,j = 1}^n F_{,ij}(\phi'\lambda_iu_i - \psi'u_i)(\phi'\lambda_ju_j - \psi'u_j) \nonumber\\
    & + \phi'\sum_{i = 1}^n F_{,ii}\lambda_i^2 + \phi''\sum_{i,j = 1}^n F_{,ij}\lambda_i\lambda_ju_iu_j - \psi'\sum_{i = 1}^n F_{,ii}\lambda_i - \psi''\sum_{i,j = 1}^n F_{,ij}u_iu_j \nonumber\\
    & \geq -C_7\frac{\lambda_1 + A_2}{\lambda_1} - C_{13}F_{,11}(\phi')^2K\lambda_1 - C_{14}F_{,11}(\psi')^2K - \varepsilon\sum_{i,j = 1}^n F_{,ij}(\phi'\lambda_iu_i - \psi'u_i)(\phi'\lambda_ju_j - \psi'u_j) \nonumber\\
    & + F_{,11}\phi'\lambda_1^2 + \phi''\sum_{i,j = 1}^n F_{,ij}\lambda_i\lambda_ju_iu_j + \kappa\psi' - \psi''\sum_{i,j = 1}^n F_{,ij}u_iu_j \textrm{ by \eqref{EQ:Bornes phi}, \eqref{EQ:Bornes psi} and \eqref{EQ:Guan}}, \nonumber\\
    & = -C_7\frac{\lambda_1 + A_2}{\lambda_1} + \kappa\psi' + F_{,11}\left(\phi'\lambda_1^2 - C_{13}(\phi')^2K\lambda_1 - C_{14}(\psi')^2K\right) \nonumber\\
    & + \sum_{i,j = 1}^n F_{,ij}u_iu_j\left((\phi'' - \varepsilon(\phi')^2)\lambda_i\lambda_j + \varepsilon\phi'\psi'(\lambda_i + \lambda_j) + (-\psi'' - \varepsilon(\psi')^2)\right) \nonumber\\
    & \geq -C_7\frac{\lambda_1 + A_2}{\lambda_1} + \kappa A_2 \label{EQ:Terme à borner 2}\\
    & + F_{,11}\left(\frac{1}{A_1K}\lambda_1^2 - \frac{4C_{13}A_1^2}{K}\lambda_1 - 4C_{14}A_2^2K\right) \label{EQ:Terme à borner 3}\\
    & + \sum_{i,j = 1}^n F_{,ij}u_iu_j\left((\phi'' - \varepsilon(\phi')^2)\lambda_i\lambda_j + \varepsilon\phi'\psi'(\lambda_i + \lambda_j) + (-\psi'' - \varepsilon(\psi')^2)\right) \textrm{ by \eqref{EQ:Bornes phi} and \eqref{EQ:Bornes psi}}. \label{EQ:Terme à borner 4}
\end{align}

\noindent\textbf{Step 8:} Bounding \eqref{EQ:Terme à borner 2} and \eqref{EQ:Terme à borner 3}.

We choose
$$
A_2 = \frac{2C_7}{\kappa},
$$
so
$$
-C_7\frac{\lambda_1 + A_2}{\lambda_1} + \kappa A_2 = C_7\left(1 - \frac{2C_7}{\kappa\lambda_1}\right),
$$
hence the positivity of \eqref{EQ:Terme à borner 2} as long as $\lambda_1 > \frac{2C_7}{\kappa}$, which we may assume without loss of generality.

Similarly, we easily see that \eqref{EQ:Terme à borner 3} is non-negative if $\frac{\lambda_1}{K}$ is large enough (in function of $A_1$ and $A_2$).

\noindent\textbf{Step 9:} Bounding \eqref{EQ:Terme à borner 4}.

For all indices $i$ and $j$, we have
\begin{align}
    & \quad\ (\phi'' - \varepsilon(\phi')^2)\lambda_i\lambda_j + \varepsilon\phi'\psi'(\lambda_i + \lambda_j) + (-\psi'' - \varepsilon(\psi')^2) \nonumber\\
    & \geq (A_1 - 1)\varepsilon(\phi')^2\lambda_i\lambda_j + \varepsilon\phi'\psi'(\lambda_i + \lambda_j) + \left(\frac{1}{4A_2\norme{u}_{\cC^0}} - \varepsilon\right)(\psi')^2 \textrm{ by \eqref{EQ:Bornes phi} and \eqref{EQ:Bornes psi}}, \nonumber\\
    & = (A_1 - 1)\left(\phi'\lambda_i + \frac{\varepsilon\psi'}{A_1 - 1}\right)\left(\phi'\lambda_j + \frac{\varepsilon\psi'}{A_1 - 1}\right) + \left(\frac{1}{4A_2\norme{u}_{\cC^0}} - \varepsilon - \frac{\varepsilon^2}{A_1 - 1}\right)(\psi')^2 \nonumber\\
    & \geq (A_1 - 1)\left(\phi'\lambda_i + \frac{\varepsilon\psi'}{A_1 - 1}\right)\left(\phi'\lambda_j + \frac{\varepsilon\psi'}{A_1 - 1}\right), \label{EQ:Inégalité estimées C2 8}
\end{align}
as long as $\varepsilon$ is small enough. Notice that $\varepsilon$ depends linearly $\frac{1}{\norme{u}_{\cC^0} + 1}$, thus $A_1$ depends quadratically in $\norme{u}_{\cC^0} + 1$ but this quadratic term can be absorbed in \eqref{EQ:lambda_max en fonction de lambda_1} in the exponential so we will eventually obtain a bound on $\frac{\lambda_{\max}}{K}$ which is exponential in $\norme{u}_{\cC^0} + 1$.

Since the matrix $(F_{,ij}u_iu_j)$ is positive, we deduce from \eqref{EQ:Inégalité estimées C2 8} that \eqref{EQ:Terme à borner 4} is non-negative.\\

\noindent\textbf{Step 10:} Conclusion.

We showed in steps 7, 8 and 9 that \eqref{EQ:Terme à borner 2} is positive and \eqref{EQ:Terme à borner 3} and \eqref{EQ:Terme à borner 4} are non-negative if $\frac{\lambda_1}{K}$ is large enough, which contradicts the non-negativity of $\sum_{i,j = 1}^n F_{,ij}\tilde{f}_{ij}$. We deduce the desired bound by contradiction, and we conclude with \eqref{EQ:Norme en fonction de lambda_1} and \eqref{EQ:lambda_max en fonction de lambda_1}.
\end{proof}

\subsection{\texorpdfstring{$\cC^{2,\alpha}$}{C2a} estimates}\label{SEC:Estimées C2alpha}

In this sub-section, we only require point 3 of Lemma \ref{LEM:Conséquences hypothèse} to hold and we don't require $0$ to be a sub-solution. We introduce
$$
M_{13}' = \sup\left\{\frac{1}{(\norme{A} + 1)\scal{\pi_{\End_S(T_xX,g)}\NN_\bH(x,A)}{\Id}} \ \middle|\ (x,A) \in \End_S(TX,g), \norme{A} \leq t\right\}.
$$
It is a well-defined number because by point 3 of Lemma \ref{LEM:Conséquences hypothèse}.

Once we have a bound on the $\cC^2$-norm, the $\cC^{2 + \alpha}$ norms are controlled thanks to the Evans--Krylov theory \cite{Evans}\cite{Krylov}. More precisely, we will use the estimate of Caffarelli and Cabré \cite[Theorem 8.1]{Caffarelli_Cabré} for elliptic equations that depend on the variable $x$. Since it requires uniform ellipticity, we introduce the following positive non-increasing function,
$$
\nu : \fonction{\R_+}{\R_+^*}{t}{\min\{\lambda_{\min}(\NN_{H_x}(A)) \mid (x,A) \in \End_S(TX,g), \norme{A} \leq t\}}.
$$

\begin{proposition}[$\cC^{2,\alpha}$ \textit{a priori} estimates]\label{PRO:Estimées C2alpha}
    There exist positive constants $\alpha = \alpha(X,g,\nu(\norme{\nabla^2u}_{\cC^0}))$ and $C = C(X,g,M_{13}',\nu(\norme{\nabla^2u}_{\cC^0})))$ such that
    $$
    \norme{u}_{\cC^{2,\alpha}} \leq C(\norme{u}_{\cC^0} + 1).
    $$
\end{proposition}
\begin{proof}
We use \cite[Theorem 8.1]{Caffarelli_Cabré}. Since $\norme{\nabla^2u(x)}$ at each point is bounded by $\norme{\nabla^2u}_{\cC^0}$, we have for all $x$,
$$
\lambda_{\inf}(\nabla_AF(x,\nabla^2u(x))) = \frac{\lambda_{\inf}(\NN_{H_x}(\nabla^2u(x)))}{\norme{\nabla_AF(x,\nabla^2u(x))}} \geq \nu(\norme{\nabla^2u}_{\cC^0}),
$$
and
$$
\lambda_{\max}(\nabla_AF(x,\nabla^2u(x))) \leq \tr(\nabla_AF(x,\nabla^2u(x))) = 1.
$$
We deduce that the equation $F(x,\nabla^2u(x)) = 0$ is uniformly elliptic in its second variable with constants of uniform ellipticity that depend only on $\nu(\norme{\nabla^2u}_{\cC^0})$.

Fix a point $x_0 \in X$. By working in a geodesic ball $B(x_0,r)$ of small enough radius, we may identify $T_xX$ with $\R^n$ for $x$ close enough to $x_0$. Using Caffarelli and Cabré's notations, for all $x \in B(x_0,r)$,
$$
\tilde{\beta}(x) = \sup_{A \in \End_S(\R^n)} \frac{\abs{F(x,A) - F(x_0,A)}}{\norme{A} + 1} \leq \sup_{(y,A) \in \End_S(TB(x_0,r),g)}\frac{\norme{\nabla_xF(y,A)}}{\norme{A} + 1}d(x,x_0) \leq M_{13}'d(x,x_0).
$$
In particular, $\tilde{\beta}$ is bounded and Lipschitz at $x_0$ so we can conclude by \cite[Theorem 8.1]{Caffarelli_Cabré}.
\end{proof}

\section{Existence and uniqueness of solutions}\label{SEC:Méthode continuité}

\subsection{Uniqueness}

We start by showing an easy uniqueness result. Here, we don't require any assumption on $\bH$ other than being a smooth $\End_S^{++}(TX,g)$-elliptic convex hypersurface of $\End_S(TX,g)$.

\begin{theorem}[Uniqueness of solutions]\label{THE:Unicité}
    Let $\bH^1$ and $\bH^2$ be smooth $\End_S^{++}(TX,g)$-elliptic convex hypersurfaces of $\End_S(TX,g)$ such that $\bH^1 \subset (\bH^2)^-$. Let $u_1$ and $u_2$ be smooth functions that satisfy for each $i \in \{1,2\}$,
    $$
    \forall x \in X, \quad \nabla^2u_i(x) \in H_x^i.
    $$
    Then, $u_1 - u_2$ is constant.
\end{theorem}
\begin{proof}
Assume up to a translation that $u_1 = 0$ and set $u = u_2$. Let
$$
\bH_0 = \{(x,A) \in \End_S(TX,g) \mid \scal{\NN_{H_x^1}(0)}{A} = 0\}.
$$
It is a family of $\End_S^{++}(TX,g)$-elliptic convex hypersurfaces as in Sub-section \ref{SEC:Estimées C0}. Let $\square$ be the associated linear elliptic operator. We know (see the proof of Proposition \ref{PRO:Estimées L1}) that $\ker(\square) = \R 1$ and $\ker(\square) = \R k$ for some smooth positive function $k : X \rightarrow \R$.

Moreover, for all $x$, $\nabla^2u(x) \in H_x^2 \subset (H_x^1)^+ \subset H_{0,x}^+$ so $\square u(x) \geq 0$ where $\NN_{0,x}$ is the normal vector to $H_{0,x}$. We deduce that
$$
0 \leq \int_X \square u(x)k(x) \, d\mu_g(x) = \int_X u(x)\square^*k(x) \, d\mu_g(x) = 0.
$$
Therefore, the inequality must be an equality so $\square u = 0$ \textit{i.e.} $u$ is constant.
\end{proof}

\subsection{Proof of Theorem \ref{THEI:Existence solution 1}}

Let $\bH$ be an almost spectral family of $\End_S^{++}(TX,g)$-elliptic convex hypersurfaces of $\End_S(TX,g)$ with tuple $\underline{M}$. Let $\bH_0 \subset \bH^-$ as in Sub-section \ref{SEC:Estimées C0}. To prove the existence theorem, we start with a lemma that will help us get uniform bounds.
\begin{lemma}\label{LEM:Translation hypersurface presque spectrale}
    For all smooth functions $h : X \rightarrow \R$, the family of $\End_S^{++}(TX,g)$-elliptic convex hypersurfaces defined by
    $$
    \bH' = \{(x,A) \in \End_S(TX,g) \mid F(x,A) = h(x)\},
    $$
    is almost spectral with a tuple that only depends on $\underline{M}$ and $\norme{h}_{\cC^2}$. Moreover,
    $$
    \bH_0 - \norme{\nabla^2h}_{\cC^0}\Id \subset \bH^-.
    $$
\end{lemma}
\begin{proof}
Let $(x_0,A_0) \in \bH'$ with $\norme{A_0} \geq \max\{M_1,1\} + 2\sqrt{n}\norme{h}_{\cC^0}$. Let $A_1 = A_0 - F(x_0,A_0)\Id \in \bH$. Since $F(x_0,A_0) = h(x_0)$, we have $\norme{A_1} \geq M_1$ (and $\norme{A_1} \geq 1$) so Definition \ref{DEF:Presque spectrale} gives us functions $f$, $B$ and $G$ such that locally around $(x_0,A_1)$,
$$
\bH = \{(x,A) \in \End_S(TX,g) \mid f(\lambda(B(x) + A)) + G(x,A) = 0\}.
$$
Let the germs
$$
\tilde{B} : x \longmapsto B(x) + h(x)\Id, \qquad \tilde{G} : (x,A) \longmapsto G(x,A + h(x)\Id),
$$
where $\tilde{B}$ is defined near $x_0$ and $\tilde{G}$ is defined near $(x_0,A_0)$. We keep the same function $f$ for $\bH'$.

We have
$$
\tilde{B}(x_0) + A_1 = B(x_0) + h(x_0)\Id + A_0 - F(x_0,A_0)\Id = B(x_0) + A_0,
$$
so $\tilde{B}(x_0) + A_1 \in \Gamma$ and
$$
\tilde{G}(x_0,A_1) = G(x_0,A_1 + h(x_0)\Id) - h(x_0) = G(x_0,A_0) = -f(\lambda(\tilde{B}(x_0) + A_1)).
$$
In particular, the first four points in Definition \ref{DEF:Presque spectrale} are verified. The fifth point is straightforward with in particular the same constant $M_7$. For the bound on $\norme{\nabla_A\tilde{G}(x_0,A_1)}$, we just need to notice that $\frac{1}{\norme{A_0}} \leq \frac{2}{\norme{A_1}}$. Similarly, the sixth point is straightforward with in particular the same constant $M_9$, by noticing that $\norme{A_0} \leq 2\norme{A_1}$ and $\norme{A_1} \geq 1$.

The statement about $\bH_0 - \norme{\nabla^2h}_{\cC^0}\Id$ is trivial.
\end{proof}

Assume now that $\bH$ satisfies the property stated in Theorem \ref{THEI:Existence solution 1}. We set $F = F_\Id$. Assume up to a translation that the $\cC$-sub-solution is $\underline{u} = 0$. We set up a continuity method. Let for all $t \in [0,1]$,
$$
(E_t) : \quad \forall x \in X, \quad F(x,\nabla^2u(x)) = (1 - t)F(x,0).
$$
$$
I = \{t \in [0,1] \mid (E_t) \textit{ admits a smooth solution}\}.
$$

Clearly, $0 \in I$ because $u = 0$ solves $(E_0)$.

We then show that $I$ is open. Notice first that by assumption on $\bH$ (see the statement of Theorem \ref{THEI:Existence solution 1}) with $u = 0$ and $h(x) = F(x,0)$, we must have
\begin{equation}\label{EQ:Intégrale nulle}
    \int_X F(x,0) \, d\mu_g(x) = 0.
\end{equation}
Then, let $t_0 \in I$ and $u_0$ a smooth solution of $(E_{t_0})$ such that $\int_X u(x) \, d\mu_g(x) = 0$. We denote by $\cC_0^{k,\alpha}(X,\R)$ the Hölder space of $\cC^{k,\alpha}$ functions from $X$ to $\R$ whose integral over $X$ vanishes. Let
$$
\Psi : \fonction{I \times \cC_0^{2,1/2}(X,\R)}{\cC_0^{0,1/2}(X,\R)}{(t,u)}{x \mapsto F(x,\nabla^2u(x)) - (1 - t)F(x,0) - \frac{1}{\Vol(X)}\int_X F(y,\nabla^2u(y)) d\mu_g(y)}.
$$
By \eqref{EQ:Intégrale nulle}, $\Psi$ indeed takes values in $\cC_0^{0,1/2}(X,\R)$. We have $\Psi(t_0,u_0) = 0$ and
$$
\frac{\partial\Psi}{\partial u}(t_0,u_0)\hat{u} : x \mapsto \scal{\nabla_AF(x,\nabla^2u_0(x))}{\nabla^2\hat{u}(x)} - \frac{1}{\Vol(X)}\int_X \scal{\nabla_AF(y,\nabla^2u_0(y))}{\nabla^2\hat{u}(y)} d\mu_g(y).
$$
By ellipticity of $F$, this linear map is an isomorphism. By the implicit function theorem, for all $t$ close enough to $t_0$, there is a $u_t \in \cC_0^{2,1/2}(X,\R)$ such that for all $x$,
$$
F(x,\nabla^2u_t(x)) = (1 - t)F(x,0) + \frac{1}{\Vol(X)}\int_X F(y,\nabla^2u_t(y)) d\mu_g(y).
$$
In other words,
$$
\nabla^2u_t(x) - \left((1 - t)F(x,0) + \frac{1}{\Vol(X)}\int_X F(y,\nabla^2u_t(y)) d\mu_g(y)\right)\Id \in \bH.
$$
By assumption on $\bH$ and \eqref{EQ:Intégrale nulle}, we must have
$$
\int_X F(y,\nabla^2u_t(y)) \, d\mu_g(y) = 0.
$$
We deduce that $u_t$ is a solution of $(E_t)$ hence the openness of $I$.

Let $t_\infty \in \overline{I}$ and $(t_m)_{m \geq 0}$ a sequence of elements of $I$ that converges toward $t_\infty$. We want to show that $t_\infty \in I$. Let for all $m$, $u_m$ be the solution of $(E_{t_m})$ such that $\max_{x \in X} u_m(x) = 0$.

By Lemma \ref{LEM:Translation hypersurface presque spectrale}, each
$$
\bH^t = \{(x,A) \in \End_S(TX,g) \mid F(x,A) = (1 - t)F(x,0)\}
$$
is almost spectral with a tuple that is independent from $t$. Moreover, there is a family $\bH_0$ as in Sub-section \ref{SEC:Estimées C0} independent from $t$ such that $\bH_0 \subset (\bH^t)^-$. We easily see that the constant $M_{13}'$ and the function $\nu$ of Sub-section \ref{SEC:Estimées C2alpha} can be controlled uniformly with respect to $t$.

Finally, since $0$ is a $\cC$-sub-solution for $\bH$, let $\delta$ and $R$ be positive constants such that \eqref{EQ:Définition delta et R} holds for all $x$. Moreover, $0$ is a solution thus a $\cC$-sub-solution for $\bH^0$ so up to reducing $\delta$ and increasing $R$, we may assume that \eqref{EQ:Définition delta et R} also holds for $\bH^0$. Let $t \in [0,1]$ and let us denote by $F_P^t$ the functions defined in Sub-section \ref{SEC:Fonction F_P} with respect to $\bH^t$. In particular, $F^1 = F$. For all $x$ and all $P \in \End_S^{++}(T_xX,g)$ of norm $1$, $-\delta\Id - F_P^t(x,-\delta\Id)P \in \bH^t$ so $-\delta\Id - (1 - t)F_\Id^1(x,0)\Id - F_P^t(x,-\delta\Id)P \in \bH$. We deduce that
$$
F_P^t(x,-\delta\Id) = F_P^1(x,-\delta\Id - (1 - t)F_\Id^1(x,0)\Id).
$$
If $0 \in H_x^-$, then $F_\Id^1(x,0) \leq 0$ so
$$
-F_P^t(x,-\delta\Id) = -F_P^1(x,-\delta\Id - (1 - t)F_\Id^1(x,0)\Id) \leq -F_P^1(x,-\delta\Id) \leq R.
$$
Else, $0 \in H_x^+$. We easily see that it implies that $H_x^t \subset (H_x^0)^-$ for all $t$ hence
$$
-F_P^t(x,-\delta\Id) \leq -F_P^0(x,-\delta\Id) \leq R.
$$
In any case, \eqref{EQ:Définition delta et R} holds for all $t$ at all $x$ with a uniform $\delta$ and $R$.

By Propositions \ref{PRO:Estimées L1}, \ref{PRO:Estimées C0}, \ref{PRO:Estimées C1} and \ref{PRO:Estimées C2alpha} and Theorem \ref{THE:Estimées C2}, there is a uniform $\alpha > 0$ and a uniform \textit{a priori} estimate on the $\cC^{2,\alpha}$ bound of all $u_m$, say
$$
\norme{u_m}_{\cC^{2,\alpha}} \leq C.
$$

By compactness of the closed balls of $\cC^{2,\alpha}$ in $\cC^{2,\alpha/2}$, we may assume that $u_m \rightarrow u_\infty$ in $\cC^{2,\alpha/2}$ so $u_\infty$ is a solution to $(E_{t_\infty})$. By a standard bootstrapping argument $u_\infty$ is smooth thus $t_\infty \in I$. We deduce that $I = [0,1]$ so $(E_1)$ admits a smooth solution, which proves Theorem \ref{THEI:Existence solution 1}. The uniqueness is a consequence of Theorem \ref{THE:Unicité}.

\subsection{Proof of Theorem \ref{THEI:Existence solution 2}}

Let $\boldsymbol{\cH}$ as in the statement of Theorem \ref{THEI:Existence solution 2}. For all $c \in [a,b]$, let $F^c$ be the function $F_\Id$ associated with $\bH^c$. Assume up to a translation that the $\cC$-sub-solution for $\bH^b$ is $\underline{u} = 0$. We set up a continuity method. Let for all $t \in [0,1]$ and $c \in [a,b]$,
$$
(E_{t,c}) : \quad \forall x \in X, \quad F^c(x,\nabla^2u(x)) = (1 - t)F^c(x,0) + K,
$$
where the unknown is the couple $(u,K)$ where $K \in \R$ is a constant. Let
$$
I = \{(t,c) \in [0,1] \times [a,b] \mid (E_{t,c}) \textit{ admits a solution $(u,K)$ with $u$ smooth}\}.
$$

Clearly, for all $c$, $(0,c) \in I$ because $(u,K) = (0,0)$ solves $(E_{0,c})$.

We then show that $I$ is open. Let $(t_0,c_0) \in I$, $(u_0,K_0)$ a solution of $(E_{t_0,c_0})$ and
$$
\Psi : \fonction{I \times \cC_0^{2,1/2}(X,\R)}{\cC_0^{0,1/2}(X,\R)}{((t,c),u)}{x \mapsto F^c(x,\nabla^2u(x)) - (1 - t)F^c(x,0) - \scriptstyle\frac{1}{\Vol(X)}\int_X F(y,\nabla^2u(y)) - (1 - t)F^c(x,0) d\mu_g(y)}.
$$
We have $\Psi((t_0,c_0),u_0) = 0$ and
$$
\frac{\partial\Psi}{\partial u}((t_0,c_0),u_0)\hat{u} : x \mapsto \scal{\nabla_AF^{c_0}(x,\nabla^2u_0(x))}{\nabla^2\hat{u}(x)} - \frac{1}{\Vol(X)}\int_X \scal{\nabla_AF^{c_0}(y,\nabla^2u_0(y))}{\nabla^2\hat{u}(y)} d\mu_g(y).
$$
By ellipticity of $F$, this linear map is an isomorphism. By the implicit function theorem, for all $t$ close enough to $t_0$, there is a $u_{t,c} \in \cC_0^{2,1/2}(X,\R)$ such that for all $x$,
$$
F^c(x,\nabla^2u_t(x)) = (1 - t)F^c(x,0) + \frac{1}{\Vol(X)}\int_X F^c(y,\nabla^2u_t(y)) - (1 - t)F^c(x,0) d\mu_g(y).
$$
It shows that $(t,c) \in I$ hence the openness of $I$.

The closedness part will be a bit different than for the proof of Theorem \ref{THEI:Existence solution 1}. We start with a simple lemma on the constant $K$.

\begin{lemma}\label{LEM:Borne K}
    If $(u,K)$ is a solution to $(E_{t,c})$, then
    $$
    \min_{x \in X} F^b(x,0) \leq t\min_{x \in X} F^c(x,0) \leq K \leq t\max_{x \in X} F^c(x,0) \leq \max_{x \in X} F^a(x,0).
    $$
\end{lemma}
\begin{proof}
Let $x_0$ be a point where $u$ reaches its maximum. In particular, $\nabla^2u(x_0) \leq 0$ so by ellipticity, $F^c(x_0,\nabla^2u(x_0)) \leq F^c(x_0,0)$ hence
$$
K = F^c(x_0,\nabla^2u(x_0)) - (1 - t)F^c(x_0,0) \leq tF^c(x_0,0) \leq t\max_{x \in X} F^c(x,0).
$$
Moreover, since $\bH^c \subset (\bH^a)^+$, $F^c \leq F^a$ and since for all $x$, $(x,0) \in (\bH^a)^+$, $F^a(x,0) \geq 0$ hence
$$
t\max_{x \in X} F^c(x,0) \leq \max_{x \in X} F^a(x,0).
$$
We get the other inequality similarly.
\end{proof}

$$
\bH^{t,c,K} = \{(x,A) \in \End_S(TX,g) \mid F^c(x,A) = (1 - t)F^c(x,0) + K\},
$$
are almost spectral with a uniform tuple $\underline{M}$ when $K$ stays in the compact interval\\
$\left[\min_{x \in X} F^b(x,0),\max_{x \in X} F^c(x,0)\right]$. Moreover, there is a family $\bH_0$ as in Sub-section \ref{SEC:Estimées C0} independent from $(t,c,K)$ such that $\bH_0 \subset (\bH^{t,c,K})^-$. We easily see that the constant $M_{13}'$ and the function $\nu$ of Sub-section \ref{SEC:Estimées C2alpha} can be controlled uniformly with respect to $t$, $c$ and $K$.

Let $F_P^{t,c,K}$ the functions associated with $\bH^{t,c,K}$ as before. Since $0$ is a $\cC$-sub-solution for $\bH^b$, let $\delta$ and $R$ be positive constants such that \eqref{EQ:Définition delta et R} holds with the $F_P^{0,b,0}$ for all $x$. For all $c \leq b$, $\bH^c \subset (\bH^b)^-$ so \eqref{EQ:Définition delta et R} also holds with $F_P^{0,c,0}$. With a similar reasoning as in the proof of Theorem \ref{THEI:Existence solution 1}, \eqref{EQ:Définition delta et R} also holds with $F_P^{t,c,0}$ (up to changing $\delta$ and $R$ but they remain uniform). Finally, if $K \leq 0$, then $\bH^{t,c,K} \subset (\bH^{t,c,0})^-$ so \eqref{EQ:Définition delta et R} remains true for $F_P^{t,c,K}$.

By Propositions \ref{PRO:Estimées L1}, \ref{PRO:Estimées C0}, \ref{PRO:Estimées C1} and \ref{PRO:Estimées C2alpha} and Theorem \ref{THE:Estimées C2}, there is a uniform $\alpha > 0$ and a uniform \textit{a priori} estimate on the $\cC^{2,\alpha}$ bound on all $u : X \rightarrow \R$ such that $(u,K)$ is a solution of some $(E_{t,c})$, as long as $K \leq 0$ (and we still normalise with $\max_{x \in X} u(x) = 0$). With the same argument as before, if $(t_m,c_m)$ is a sequence of elements of $I$ that converges toward some $(t_\infty,c_\infty)$ and $(u_m,K_m)$ is a sequence of solutions of $E_{t_m,c_m}$ with $u_m$ normalised and $K_m \leq 0$, then $(t_\infty,c_\infty) \in I$.

We are now ready to conclude. Notice first that by Theorem \ref{THE:Unicité}, the (normalised) solution $(u_{t,c},K_{t,c})$ of $(E_{t,c})$ is unique when it exists. Moreover, as a consequence of the implicit theorem, $(t,c) \mapsto K_{t,c}$ is continuous.

Let
$$
J_1 = \{t \in [0,1] \mid (t,b) \in I\}.
$$
$(0,b) \in I$ so $0 \in J_1$ and $J_1$ is open because $I_1$ is. Moreover, recall that by assumption, for all $x \in X$, $(x,0) \in (H_x^b)^-$ so $F^b(x,0) \leq 0$. By Lemma \ref{LEM:Borne K}, we deduce that for all $t \in J_1$, $K_{t,b} \leq 0$. Therefore, $J_1$ is closed so $J_1 = [0,1]$.

Let
$$
J_2 = \{c \in [a,b] \mid (1,c) \in I\}.
$$
$1 \in J_1$ so $b \in J_2$ and $J_2$ is open by openness of $I$. Let $c_0 \in [a,b]$ be the smallest possible such that $]c_0,b] \subset J_2$. We have to distinguish two cases.

\textbf{First case:} For all $c > c_0$, $K_{1,c} \leq 0$. In this case, $c_0 \in J_2$ so by openness of $J_2$ and minimality of $c_0$, we must have $c_0 = a$ and $J_2 = [a,b]$. By assumption, for all $x$, $(x,0) \in (\bH^a)^+$ so $F^a(x,0) \geq 0$. By Lemma \ref{LEM:Borne K}, we must have $K_{1,a} \geq 0$ but $K_{1,a} \leq 0$ by continuity hence $K_{1,a} = 0$ and $u_{1,a}$ is the wanted solution with $c = a$. Notice that in this case, we must have $F(x,0) = 0$ for all $x$ \textit{i.e.} $u_{1,a} = 0$.

\textbf{Second case:} There is $c_1 > c_0$ such that $K_{1,c_1} \geq 0$. Then, since $K_{1,b} \leq 0$ and by continuity of $c \mapsto K_{1,c}$, there is a $c \geq c_1$ such that $K_{1,c} = 0$. $u_{1,c}$ is our wanted solution.\\

The uniqueness part easily follows from Theorem \ref{THE:Unicité}.

\bibliographystyle{plain}
\bibliography{bibliography}

\end{document}